\documentclass[a4paper,12pt,twoside]{article}
\usepackage{amsthm}
\usepackage{hyperref}
\usepackage[utf8]{inputenc}
\usepackage{extpfeil}
\usepackage[
backend=biber,
style=numeric
]{biblatex}
\DeclareNameAlias{default}{family-given}
\usepackage{subcaption}
\usepackage[top=2cm,bottom=2cm,left=2cm,right=2cm,asymmetric]{geometry}
\usepackage{lipsum}
\usepackage{tikz}
\usetikzlibrary{decorations.pathreplacing,shapes.misc}
\usepackage{pgfplots}
\usepgfplotslibrary{fillbetween}
\pgfplotsset{compat=1.11}
\usepackage{mathtools}
\usepackage{changepage}
\usepackage{theoremref}
\usepackage{xfrac}
\usepackage{enumitem}

\usepackage{ulem}
\newlength{\abstractwidth}
\usepackage{titlesec}
\titleformat{\section}
  {\normalfont\rmfamily\normalsize\bfseries \upshape \centering \color{black}}
  {\thesection}{1em}{}

  \titleformat{\subsection}
  {\normalfont\rmfamily\normalsize\bfseries \color{black}}
  {\thesubsection}{1em}{}

\usepackage{fancyhdr}

\renewcommand{\thanks}[1]{\footnote{#1}} 

\newcommand{\be}{\begin{equation}}
\newcommand{\bea}{\begin{eqnarray}}
\newcommand{\eea}{\end{eqnarray}} \newcommand{\ee}{\end{equation}}
 
 \def\ba{\begin{eqnarray}}
\def\ea{\end{eqnarray}}

\def\[{{\bf [}}
\def\]{{\bf ]}}

\begin{document}
\pagestyle{fancy}
\fancyhf{} 
\renewcommand{\headrulewidth}{0pt}
\fancyhf[EHC]{B. Harvie \& Y.K. Wang}
\fancyhf[OHC]{A Rigidity Theorem for A.F. Static Manifolds and its Applications}
\fancyhf[FC]{\thepage}

\newtheorem{theorem}{Theorem} [section]
\newtheorem{proposition}[theorem]{Proposition} 
\newtheorem{lemma}[theorem]{Lemma} 
\newtheorem{corollary}[theorem]{Corollary} 
\newtheorem{definition}[theorem]{Definition} 
\newtheorem{conjecture}[theorem]{Conjecture} 
\newtheorem{example}[theorem]{Example} 
\newtheorem{claim}[theorem]{Claim} 
\newtheorem{remark}[theorem]{Remark} 
\newtheorem*{proposition*}{Proposition}

\begin{centering}
 
\textup{\Large \bf A Rigidity Theorem for Asymptotically Flat Static Manifolds and its Applications}

\vspace{10 mm}
\textnormal{ \large Brian Harvie and Ye-Kai Wang }
\vspace{0.5in}
\begin{abstract}
In this paper, we study the Minkowski-type inequality for asymptotically flat static manifolds $(M^{n},g)$ with boundary and with dimension $n<8$ that was established by McCormick in \cite{static_minkowski}. First, we show that any asymptotically flat static $(M^{n},g)$ which achieves the equality and has CMC or equipotential boundary is isometric to a rotationally symmetric region of the Schwarzschild manifold. Then, we apply conformal techniques to derive a new Minkowski-type inequality for the level sets of bounded static potentials. Taken together, these provide a robust approach to detecting rotational symmetry of asymptotically flat static systems.

As an application, we prove global uniqueness of static metric extensions for the Bartnik data induced by both Schwarzschild coordinate spheres and Euclidean coordinate spheres in dimension $n < 8$ under the natural condition of \textit{Schwarzschild stability}. This generalizes an earlier result of Miao in \cite{boundary_effects}. We also establish uniqueness for equipotential photon surfaces with small Einstein-Hilbert energy. This is interesting to compare with other recent uniqueness results for static photon surfaces and black holes, e.g. in \cite{level_set_static}, \cite{photon_surface_equipotential}, and \cite{photon_sphere_spinorial}.
\end{abstract}



\end{centering}

\section{Introduction} 

A Riemannian manifold $(M^{n},g)$ is \textit{static} if it admits a positive solution $V: M \rightarrow \mathbb{R}^{+}$ to the system of equations

\begin{eqnarray} \label{static_equations}
\nabla^{2}_{g} V &=& V \text{Ric}_{g}, \\
\Delta_{g} V &=& 0, \nonumber
\end{eqnarray}
where $\nabla^{2}_{g}$, $\Delta_{g}$, and $\text{Ric}_{g}$ are the Hessian, Laplacian, and Ricci tensor of $(M^{n},g)$, respectively. \eqref{static_equations} are known as the \textit{static equations}, and the solution $V$ is called a \textit{static potential}. These names come from the role $g$ and $V$ play in general relativity. Specifically, $(M^{n},g)$ is static with static potential $V$ if and only the time-independent Lorentzian metric

\begin{equation} \label{static_product}
  \overline{g} = -V(x)^{2} dt^{2} + g, \hspace{1cm} x \in M^{n}, 
\end{equation}
solves the vacuum Einstein field equations $\text{Ric}_{\overline{g}} =0$ on $L^{n+1}= M^{n} \times \mathbb{R}$. The Lorentzian manifold $(L^{n+1},\overline{g})$ is ``static" because the time-translation vector field $\partial_{t}$ on $L^{n+1}$ is a timelike, hypersurface orthogonal Killing vector field of $(L^{n+1},\overline{g})$.

The most important static manifold is the \textit{Schwarzschild manifold}, parametrized by $m \in \mathbb{R}$ and equipped with the Schwarzschild potential function

\begin{eqnarray}
 (M^{n},g_{m}) &=& ( \mathbb{S}^{n-1} \times (r_{m},\infty), \frac{1}{1- \frac{2m}{r^{n-2}}} dr^{2} + r^{2} \sigma^{\mathbb{S}^{n-1}}), \hspace{0.6cm} r_{m} = (\max \{0,2m \})^{\frac{1}{n-2}} \nonumber \\
 V_{m}(r) &=& \sqrt{1- \frac{2m}{r^{n-2}}}.  \label{schwarzschild}
\end{eqnarray}
The corresponding spacetime $(L^{n+1},\overline{g}_{m})$ is the Schwarzschild spacetime, which models the gravitational field outside of a rotationally symmetric star.

A fundamental set of questions in general relativity concern whether or not the spacetime $(L^{n+1},\overline{g})$ outside of a static, isolated astrophysical object, such as a static star or black hole, is neccessarily Schwarzschild. The spacetime $(L^{n+1},\overline{g})$ has the form \eqref{static_product} in this situation, so a constant time slice $(M^{n} \times \{ t_{0} \},g)$ of $(L^{n+1},\overline{g})$ is a static Riemannian manifold with a geometric constraint imposed on its boundary by the object. Therefore, the problem is to determine if a static manifold $(M^{n},g)$ satisfying a geometric boundary condition and with suitable asymptotics is rotationally symmetric and hence isometric to a piece of \eqref{schwarzschild}. Problems like these have been studied as far back as Israel's static black hole uniqueness theorem from \cite{event_horizons_static}.

In this work, we develop a new approach to the uniqueness of asymptotically flat static systems that is based on a Minkowski-type inequality established by McCormick in \cite{static_minkowski}. First, we characterize the equality case in this inequality under suitable boundary assumptions. This provides a uniqueness criterion for asymptotically flat static manifolds of dimension $n<8$ in terms of boundary data. Then, we apply this criterion to establish uniqueness of \eqref{schwarzschild} in several important contexts. Notably, we prove global uniqueness for suitably-defined static extensions of the Bartnik data induced by Schwarzschild and Euclidean coordinate spheres when $n < 8$, as well as new uniqueness theorems for CMC spheres and for equipotential photon surfaces.  

Our approach involves several other results which are more broadly interesting. For example, we prove via conformal techniques an additional Minkowski-type inequality for the level sets of bounded static potentials which extends the Willmore-type inequalities from \cite{level_set_static}. We introduce a notion of \textit{Schwarzschild stability} for CMC hypersurfaces which is especially useful for studying CMC $2$-spheres in $3$-manifolds with non-negative scalar curvature. Finally, we derive a new upper bound on the Hawking mass of surfaces in asymptotically flat static $3$-manifolds.

\subsection{Main Results} We will wait until the next section to define our standard notation as well as the standard objects of mathematical relativity, e.g. the ADM mass, outer-minimizing hypersurfaces, etc., and so a reader unfamiliar with the field should read that section first. In all of our theorems, the static potential $V$ is bounded on $M^{n}$ and scaled so that

\begin{equation*}
    \lim_{|x| \rightarrow \infty} V(x)= 1
\end{equation*}
in an asymptotically flat coordinate chart. Our results build upon a previously-established Minkowski-type inequality for static manifolds. Recall that the Minkowski inequality for a convex hypersurface $\Sigma^{n-1} \subset \mathbb{R}^{n}$ reads

\begin{equation} \label{euclidean_minkowski}
    \frac{1}{(n-1) w_{n-1}} \int_{\Sigma} H d\sigma \geq \left(\frac{|\Sigma|}{w_{n-1}} \right)^{\frac{n-2}{n-1}},
\end{equation}
where $H$ is the mean curvature of $\Sigma$ taken with respect to the outward normal. \eqref{euclidean_minkowski} has been generalized both to non-convex hypersurfaces in Euclidean space and to hypersurfaces in non-Euclidean backgrounds, c.f. \cite{sub_static_li_xia}, \cite{minkowski_linear}, \cite{minkowski_euclidean}, \cite{minkowski_non_linear}, \cite{imcf_isoperimetric}, and \cite{ads_minkowski}. 

The inverse mean curvature flow (see Section 2.2) provides an approach to Minkowski-type inequalities in manifolds satisfying suitable asymptotic decay conditions. In \cite{schwarzschild_minkowski}, Wei proves a Minkowski inequality for hypersurfaces in the Schwarzschild manifold using a monotone quantity along the weak inverse mean curvature flow. Later in \cite{static_minkowski}, McCormick generalized Wei's inequality to asymptotically flat manifolds $(M^{n},g)$ in dimension $n < 8$ that admit a bounded static potential $V$. In this ``static Minkowski inequality", the hypersurface $\Sigma$ is a boundary component of $(M^{n},g)$, and the convexity assumption is replaced with the much weaker assumption that $\Sigma$ is outer-minimizing.

\begin{theorem}[ \cite{static_minkowski}, Static Minkowski Inequality] \thlabel{minkowski}
Let $(M^{n},g, V)$, $3 \leq n \leq 7$, be an asymptotically flat static system with ADM mass $m$ and compact, non-empty boundary $\partial M$. Suppose that $\partial M = \Sigma \cup (\cup_{i=1}^{k} \Sigma_{i})$, where $\Sigma$ is an outer-minimizing hypersurface and $\Sigma_{i}$ are closed minimal surfaces for each $i= 1, \dots, k$. Then we have the inequality

\begin{equation} \label{inequality}
    \frac{1}{(n-1)w_{n-1}} \int_{\Sigma} VH d \sigma + 2m \geq \left(\frac{|\Sigma|}{w_{n-1}} \right) ^{\frac{n-2}{n-1}}
\end{equation}
on the boundary component $\Sigma$. Furthermore, equality holds in \eqref{inequality} if and only if $\partial M = \Sigma$ and $M^{n}$ is foliated by a smooth solution $\Sigma_{t}$ to the inverse mean curvature flow where each $\Sigma_{t}$ is totally umbilical.
\end{theorem}
We would like to obtain a sharper characterization of the equality case in \eqref{inequality}. When $(M^{n},g)$ is the outside of a coordinate sphere in the Schwarzschild manifold, the inverse mean curvature flow $\Sigma_{t}$ of $\Sigma = \{r = r_{0} \}$ is a coordinate sphere of radius $r(t) \rightarrow \infty$, meaning equality holds in that setting. The first goal of this paper is to show that the equality holds \textit{only} on Schwarzschild coordinate spheres under natural boundary assumptions. By exploiting the umbilical foliation of $(M^{n},g)$ by IMCF that is noted in the rigidity statement of \thref{minkowski}, we prove this when the boundary component $\Sigma$ is assumed to either have constant mean curvature or to belong to a level set of $V$ (see the end of the section for a discussion on the general case of equality).

\begin{theorem}[Rigidity of Asymptotically Flat Static Manifolds] \thlabel{cmc_rigidity}
Let $(M^{n},g, V)$ be as in \thref{minkowski}, and suppose that either

\begin{itemize}
    \item $H_{\Sigma} = H_{0}> 0$ is constant, OR
    \item $V|_{\Sigma} = V_{0} > 0$ is constant
\end{itemize}
on the boundary component $\Sigma$. Then equality is achieved in \eqref{inequality} if and only if $(M^{n},g)$ is isometric to the exterior of the coordinate sphere $\{ r = r_{0} \}$ in the Schwarzschild manifold of mass $m$, where $r_{0}= \left(\frac{|\Sigma|}{w_{n-1}} \right)^{\frac{1}{n-1}}$.   
\end{theorem}
\begin{remark}
Rigidity in \thref{minkowski} is already understood when either $H_{0}=0$ or $V_{0}=0$ (incidentally, these are equivalent conditions for asymptotically flat systems with outer-minimizing boundary by Theorem 1 in \cite{static_min}). In those cases, $\Sigma$ is an outer-minimizing minimal hypersurface and the inequality \eqref{inequality} reduces to

\begin{equation*}
    m \geq  \frac{1}{2}\left(\frac{|\Sigma|}{w_{n-1}} \right)^{\frac{n-2}{n-1}}.
\end{equation*}
Indeed, this is the Riemannian Penrose inequality for asymptotically flat manifolds in dimensions less than $8$, which according to \cite{penrose_less_than_8} is saturated only by the Schwarzschild black hole. Therefore, we ignore these cases in this paper.
\end{remark}
Due to a well-known integral formula for the mass of an asymptotically flat static system (see Section 2.1), all terms in inequality \eqref{inequality} may be expressed as surface integrals. Thus \thref{cmc_rigidity} gives a uniqueness criterion for the Schwarzschild manifold solely in terms of boundary data. In particular, this criterion does not depend on the topology of $\Sigma$ or on the interior behavior of $V$, e.g. on the presence of critical points. However, these theorems are true only when $V$ is asymptotic to $1$, so one must invoke this asymptotic behavior in order to apply \thref{cmc_rigidity} to a given geometric boundary condition on $(M^{n},g,V)$.

To this end, we prove a new Minkowski-type inequality for asymptotically flat static systems with equipotential boundary, that is for $(M^{n},g,V)$ with $V|_{\Sigma} = V_{0} >0$ constant. \thref{minkowski} and \thref{cmc_rigidity} both apply regardless of the sign of $m$, but here we require $m>0$ as well as a connected boundary.

\begin{theorem}[Minkowski Inequality for Equipotential Boundaries] \thlabel{level_set_rig}
Let $(M^{n},g, V)$, $3 \leq n \leq 7$, be an asymptotically flat static system with ADM mass $m>0$. Suppose that $\partial M= \Sigma$ is connected and outer-minimizing. If $V|_{\Sigma}= V_{0} > 0$ is constant, then 

\begin{equation} \label{static_level_set}
    V_{0} \leq \frac{1}{(n-1)w_{n-1}} \left(\frac{|\Sigma|}{w_{n-1}} \right)^{\frac{2-n}{n-1}}  \int_{\Sigma} H d \sigma.
\end{equation}
Furthermore, equality holds in \eqref{static_level_set} if and only if $(M^{n},g)$ is isometric to the exterior of the coordinate sphere $\{ r = r_{0} \}$ in the Schwarzschild manifold of mass $m$, where $r_{0}= ( \frac{|\Sigma|}{w_{n-1}})^{\frac{1}{n-1}}$.
\end{theorem}

\begin{remark}
This inequality also applies in Euclidean space $(\mathbb{R}^{n}, \delta)$, for if we set $V \equiv 1$ we see that \eqref{static_level_set} is the Euclidean Minkowski inequality \eqref{euclidean_minkowski}.
\end{remark}
\begin{remark}
For any asymptotically flat static system $(M^{n},g,V)$, the level set $\{ V = V_{0} \} \subset M^{n}$ is automatically a smooth, connected, and outer-minimizing hypersurface whenever $V_{0}$ is sufficiently close to $1$, see Section 2.2. Therefore, \thref{level_set_rig} gives a level-set inequality for these $V_{0}$.
\end{remark}

In \cite{level_set_static}, Agostiniani and Mazzieri also consider level-set inequalities in static manifolds with mass $m >0$. In their Theorem 2.11(i), they establish a Willmore-type inequality of the form

\begin{equation} \label{willmore}
    V_{0} \leq \frac{1}{(n-1)w_{n-1}^{\frac{1}{n-1}}} \left(\int_{\Sigma} H^{n-1} d\sigma \right)^{\frac{1}{n-1}}
\end{equation}
for every $n \geq 4$ when $\Sigma= \{ V=V_{0} \}$ is a level set of the bounded static potential $V$. When $n< 8$ and $\Sigma$ is outer-minimizing, inequality \eqref{static_level_set} immediately implies \eqref{willmore} (and indeed, an analogous inequality for any $p \in (1,n-1)$) via H\"older's inequality. Moreover, \eqref{static_level_set} extends \eqref{willmore} to the $n=3$ case. In contrast to their approach, our proof of \thref{level_set_rig} does not utilize the level-set flow.

Our first application of these theorems is to the problem of uniqueness for static metric extensions. This problem is motivated by R. Bartnik's proposal of a quasi-local mass for compact, spacelike, time-symmetric submanifolds of spacetime-- see \cite{mass_3_metrics} and \cite{static_extension} for a review. Here, we would like to show that any asymptotically flat static manifold with a boundary that is both intrinsically round and CMC is rotationally symmetric. Due to work by Anderson in \cite{static_existence} and Huang-An in \cite{existence_static_extension}, we know this is true whenever this manifold is assumed to be sufficiently close to the Schwarzschild manifold in a suitable topology. In this paper, we approach the question of global uniqueness using a stability assumption on the boundary.

The stability operator $S_{\Sigma}$ of a CMC hypersurface $\Sigma^{n-1}$ of a Riemannian manifold $(M^{n},g)$ is

\begin{eqnarray} \label{stability_operator}
    S_{\Sigma}: C^{\infty}_{0} (\Sigma) \rightarrow C^{\infty}_{0}(\Sigma),  &\text{  where  }& C^{\infty}_{0}(\Sigma) = \{ \phi \in C^{\infty} (\Sigma) | \int_{\Sigma} \phi d \sigma = 0 \}, \nonumber \\ 
    S_{\Sigma} \phi &=& -\Delta_{\Sigma} \phi - (|h|^{2} + \text{Ric}(\nu,\nu)) \phi.
\end{eqnarray}
A result by Miao (\cite{boundary_effects}, Corollary 2), establishes for $n=3$ the uniqueness of static metric extensions for the Bartnik data of the Euclidean sphere under the condition $\lambda_{1}(S_{\Sigma}) \geq 0$ on the first eigenvalue of $S_{\Sigma}$. Here, we define a natural extension of this stability condition in terms of mean curvature and area to apply to more general CMC Bartnik data.

\begin{definition} \thlabel{stability}
Let $(M^{n},g)$ be a Riemannian manifold, and let $\Sigma^{n-1} \subset M^{n}$ be a CMC hypersurface with $H_{\Sigma}=H_{0}$. Write $r_{0}= \left(\frac{|\Sigma|}{w_{n-1}} \right)^{\frac{1}{n-1}}$ as before, and define the constant 

\begin{equation} \label{m_0}
    m_{0} = \frac{r_{0}^{n-2}}{2} \left( 1 - \frac{r_{0}^{2}}{(n-1)^{2}} H_{0}^{2} \right).
\end{equation}
We say that $\Sigma$ is \textbf{Schwarzschild stable} in $(M^{n},g)$ if the first eigenvalue of $S_{\Sigma}$ satisfies 
\begin{eqnarray*}
    \lambda_{1}(S_{\Sigma}) \geq \frac{n(n-1)m_{0}}{r_{0}^{n}}.
\end{eqnarray*}
\end{definition}
\begin{remark}
When $n=3$, the constant $m_{0}$ equals the Hawking mass $m_{H}(\Sigma)$ of a CMC surface $\Sigma$. \thref{stability} is motivated by a natural relationship between $\lambda_{1}(S_{\Sigma})$ and $m_{H}(\Sigma)$ that we will discuss in Section 5.
\end{remark}

$\lambda_{1}(S_{\Sigma}) = \frac{n(n-1)m_{0}}{r_{0}^{n}}$ for a coordinate sphere in a Schwarzschild manifold of radius $r_{0}$ and mean curvature $H_{0}$, so this definition requires that $\Sigma^{n-1} \subset (M^{n},g)$ is ``at least as stable" as the Schwarzschild sphere with the same mean curvature and area.

\begin{theorem}[Uniqueness of Static Extensions for Constant, Schwarzschild-Stable Bartnik Data] \thlabel{bartnik}
 Let $(M^{n},g,V)$, $3 \leq n \leq 7$, be an asymptotically flat static system with outer-minimizing boundary $\partial M = \Sigma \cong \mathbb{S}^{n-1}$. Suppose on $\Sigma$ that

 \begin{eqnarray}
    \sigma_{ij} &=& r_{0}^{2} \sigma^{\mathbb{S}^{n-1}}_{ij} \label{bartnik_data} \\
     H_{\Sigma} &=& H_{0} \nonumber
 \end{eqnarray}
 for constants $r_{0}$ and $H_{0}$ satisfying $0 < H_{0} \leq (n-1)r_{0}^{-1}$. Suppose further that $\Sigma$ is Schwarzschild stable in $(M^{n},g)$. Then $(M^{n},g)$ is isometric to the exterior of $\{ r = r_{0} \}$ in the Schwarzschild manifold of mass $m_{0}$ for $m_{0}$ defined in \eqref{m_0}.
\end{theorem}
\begin{remark}
The ADM mass for the Schwarzschild extension of the Bartnik data \eqref{bartnik_data} is $m=m_{0}$, and so data sets with $H_{0} > (n-1)r_{0}$ or equivalently $m_{0} < 0$ are not natural to study from the standpoint of the Bartnik quasi-local mass.
\end{remark}
When $n=3$, the notion of Schwarzschild stability may be applied to other uniqueness questions about CMC boundary data. For example, we find that only the Schwarzschild $3$-manifold contains an equipotential, Schwarzschild-stable CMC sphere.

\begin{theorem}[Uniqueness of Equipotential CMC Spheres in Static $3$-Manifolds] \thlabel{equi_cmc}
Let $(M^{3},g,V)$ be an asymptotically flat static system with outer-minimizing boundary $\partial M= \Sigma \cong \mathbb{S}^{2}$. Suppose that 

\begin{enumerate}
    \item $\Sigma^{2}$ is Schwarzschild-stable CMC with $m_{0} >0$, AND
    \item $V|_{\Sigma} = V_{0}$ is constant.
\end{enumerate}
Then $(M^{3},g)$ isometric to the exterior of $\{ r = r_{0} \}$ in the Schwarzschild manifold of mass $m_{0}$, where $r_{0}= (\frac{|\Sigma|}{w_{2}})^{\frac{1}{2}}$. 
\end{theorem}

The next application of these theorems is to photon surfaces, a notion originally introduced by Claudel et. al in \cite{geometry_photon_surfaces} and later given a geometric definition in \cite{photon_sphere_uniqueness}.

\begin{definition}
Let $(L^{n+1},\overline{g})$ be a Lorentzian manifold. An embedded hypersurface $P^{n} \subset L^{n+1}$ is a \textbf{photon surface} if it is timelike and totally umbilical in $(L^{n+1},\overline{g})$. 
\end{definition}
The above definition is equivalent to $P^{n}$ being null totally geodesic in $(L^{n+1}, \overline{g})$, meaning that $P^{n}$ traps all null geodesics initially tangent to it-- c.f. \cite{umbilic} and \cite{photon_sphere_uniqueness}. For example, photons in the $n+1$-dimensional, $m>0$ Schwarzschild spacetime orbit at a fixed radius $r=(nm)^{\frac{1}{n-2}}$ outside the black hole. Therefore, the surface

\begin{equation*}
    P^{n} = \{ r= (nm)^{\frac{1}{n-2}} \} \subset (\mathbb{S}^{n-1} \times (r_{m},\infty) \times \mathbb{R}, -V_{m}^{2}(r) dt^{2} + g_{m})
\end{equation*}
is a photon surface in the Schwarzschild spacetime. Whether the Schwarzschild spacetime is the only static spacetime that contains photon surfaces is an open question. Cederbaum established the first result in this direction in \cite{photon_sphere_uniqueness}, and \cite{photon_sphere_positive_mass}, \cite{photon_surface_equipotential}, \cite{photon_sphere_spinorial}, and \cite{photon_surface_perturbative} contain generalizations of her result. Currently, the most optimal photon surface uniqueness theorems for Schwarzschild from \cite{photon_surface_equipotential} and \cite{photon_sphere_spinorial} assume that $P^{n} \subset (L^{n+1}, \overline{g})$ has \textit{equipotential time slices}, i.e. that the static potential $V$ is constant over the time slice $\Sigma^{n-1}(t_{0}) =P^{n} \cap \{ t = t_{0} \} \subset M^{n} \times \{ t=t_{0}\}$ for every $t_{0} \in \mathbb{R}$.

In this paper, we derive a new (to our knowledge) uniqueness theorem for equipotential photon surfaces with small Einstein-Hilbert energy. The (normalized) Einstein-Hilbert energy of an $n-1$-dimensional Riemannian manifold $(\Sigma^{n-1},\sigma)$ is 

\begin{equation}
    E(\Sigma) = \left( |\Sigma| \right)^{\frac{3-n}{n-1}} \int_{\Sigma} R_{\sigma} d\sigma. 
\end{equation}
We write $E(\Sigma)$ in dimension $n-1$ since we consider the Einstein-Hilbert energy of a hypersurface in $(M^{n},g)$, and we denote the Einstein-Hilbert energy of the standard sphere by $E(\mathbb{S}^{n-1})$. When $n=3$, $E(\Sigma) \leq Y(\mathbb{S}^{2})$ as a trivial consequence of the Gauss-Bonnet theorem. When $n>3$, we know from seminal works of Schoen \cite{conformal_deformation} and Aubin \cite{yamabe_french} that $E(\Sigma) \leq E(\mathbb{S}^{n-1})$ whenever $\sigma$ minimizes $E$ within its conformal class. All such minimizers are constant scalar curvature metrics, but not all constant scalar curvature metrics necessarily have small Einstein-Hilbert energy-- see \cite{brendle2010recent} for a review.

We also note that the assumption in \cite{photon_surface_equipotential} and \cite{photon_sphere_spinorial} that $V$ is constant over every time of $P^{n}$ slice may be slightly weakened. When $V$ is constant on every slice $P^{n} \times \{ t= t_{0} \}$, the normal derivative $\frac{\partial V}{\partial \nu}$ of $V$ in $P^{n} \cap \{ t=t_{0}\} \subset (M^{n},g)$ must be identically constant, see Section 4 in \cite{photon_surface_equipotential}. One recovers exactly the same boundary condition as in those works when $V$ and $\frac{\partial V}{\partial \nu}$ are assumed to be constant over only one time slice. Therefore, the equipotential time slice assumption can be replaced with an assumption over a single slice. 

\begin{theorem}[Uniqueness of Equipotential Photon Surfaces with Small Einstein-Hilbert Energy] \thlabel{photon_surface}
Let $(L^{n+1}, \overline{g})= (M^{n} \times \mathbb{R}, -V^{2}(x)dt^{2} + g)$, $3 \leq n \leq 7$, be a static spacetime of the form \eqref{static_product}, where $(M^{n},g,V)$ is an asymptotically flat static system of ADM mass $m >0$. Let $P^{n} \subset (L^{n+1}, \overline{g})$ be a two-sided photon surface. Suppose there exists a $t_{0} \in \mathbb{R}$ such that the time slice $\Sigma^{n-1}= P^{n} \cap \{ t= t_{0} \} \subset M^{n} \times \{ t= t_{0} \}$ is outer-minimizing and is the only boundary component of the exterior region of $(M^{n},g)$. Suppose further that on $\Sigma^{n-1}$

\begin{enumerate}
    \item $V|_{\Sigma}= V_{0}$, AND
    \item $\frac{\partial V}{\partial \nu} = \rho_{0}$, AND
    \item $E(\Sigma) \leq E(\mathbb{S}^{n-1}) = (n-1)(n-2)w_{n-1}^{\frac{2}{n-1}}$.
\end{enumerate}
for constants $\rho_{0}$, $V_{0}$ (note that (3) is redundant when $n=3$). Then $(L^{n+1},\overline{g})$ isometrically embeds into the Schwarzschild spacetime

\begin{equation*}
    (\tilde{L}^{n+1}, \overline{g}_{m}) = (\mathbb{S}^{n-1} \times (r_{m}, \infty)  \times \mathbb{R}, -V_{m}(r)^{2}dt^{2} + g_{m}).
\end{equation*}
\end{theorem}
Interestingly, the same consideration for $E(\Sigma)$ appears in Theorem 2.10 of \cite{level_set_static}. There, the authors derive an analogous conclusion for asymptotically flat static systems with horizon boundary under the assumption $E(\Sigma) \leq E(\mathbb{S}^{n-1})$ for all $n \geq 4$. By contrast, the photon surface uniqueness and black hole uniqueness theorems in \cite{photon_surface_equipotential} and \cite{photon_sphere_spinorial} respectively feature weaker boundary conditions-- these do not require any assumptions on $E(\Sigma)$, and they also allow for photon surfaces which are not connected. However, instead of assuming that $(M^{n},g)$ is asymptotically flat, those theorems require $(M^{n},g)$ to be \textit{asymptotically isotropic}. This asymptotic assumption is the same as asymptotic flatness when $n=3$, see \cite{isotropic}. For $n \geq 4$, however, asymptotically isotropic is at least as strong of a decay assumption on $(M^{n},g)$ as asymptotically flat, and it is not known if the two conditions are equivalent. The differing boundary conditions needed in the asymptotically flat and asymptotically isotropic cases are worth noting with this in mind.

Finally, \thref{cmc_rigidity} assumes that either $H_{\Sigma}$ or $V|_{\Sigma}$ is constant, and so it is interesting to ask if the same rigidity result for Schwarzschild spheres still applies without either assumption. For $n=3$, one way to probe this is to consider the functional $Q(\Sigma)$ associated with inequality \eqref{inequality},

\begin{equation} \label{Q}
    Q(\Sigma)= \frac{1}{2} \left(\frac{|\Sigma|}{w_{2}} \right)^{\frac{1}{2}} \left[1 - \frac{1}{2w_{2}} \left(\frac{|\Sigma|}{w_{2}} \right)^{-\frac{1}{2}} \int_{\Sigma} VH d\sigma \right],
\end{equation}
which is bounded above by the mass $m$ of $(M^{3},g)$, and to compare it to the \textit{Hawking mass} of $\Sigma$. Recall that the Hawking mass of $\Sigma^{2} \subset (M^{3},g)$ is the quantity

\begin{equation} \label{hawking}
    m_{H}(\Sigma) = \frac{1}{2} \left( \frac{|\Sigma|}{w_{2}} \right)^{\frac{1}{2}} \left[1 - \frac{1}{2w_{2}} \int_{\Sigma} H^{2} d \sigma \right].
\end{equation}
From the work of Huisken and Ilmanen in \cite{imcf_penrose}, we know that $m_{H}(\Sigma)$ is also bounded above by $m$ for an outer-minimizing $\Sigma^{2}$, and equality in that case occurs only on Schwarzschild coordinate spheres. If $m_{H}(\Sigma) \geq Q(\Sigma)$ on any outer-minimizing $\Sigma^{2}$, then $Q(\Sigma)=m$ would imply that $\Sigma$ is a Schwarzschild coordinate sphere via this rigidity fact. Here, we prove that in fact the \textit{opposite} bound holds generically in asymptotically flat static $3$-manifolds; that is $Q(\Sigma) \geq m_{H}(\Sigma)$ in a broad setting. This hints that there may exist an $(M^{3},g)$ with $m= Q(\Sigma)> m_{H}(\Sigma)$, although we are uncertain on this matter. 

\begin{theorem}[Hawking Mass Comparison] \thlabel{hawking_comparison}
Let $(M^{3},g,V)$ be an asymptotically flat static system with connected, outer-minimizing boundary $\partial M= \Sigma^{2}$. Suppose further that $V|_{\Sigma^{2}} > 0$ and that $\Sigma^{2}$ is also outer-minimizing with respect to the conformal metric
\begin{equation*}
 g_{-}=V^{4}g.   
\end{equation*}
Then

\begin{equation} \label{comparison}
    Q(\Sigma) \geq m_{H}(\Sigma),
\end{equation}
for the functionals defined in \eqref{Q}, \eqref{hawking}. Furthermore, equality holds if and only if $(M^{3},g)$ is isometric to the exterior of $\{ r= r_{0} \}$ in the Schwarzschild manifold of mass $m$.
\end{theorem}
The role that the conformal metric $g_{-}$ plays here will become clear in Section 4. The condition that $\Sigma^{2}$ is $g_{-}$-outer-minimizing is automatic if $V|_{\Sigma}=V_{0}$ and $m >0$, as we will show in that section. It would be interesting to see if a weaker requirement, such as $\frac{\partial V}{\partial \nu} >0$ on $\Sigma^{2}$, also ensures this property. 

\subsection{Outline}
We now summarize the structure and approach of the paper. In Section 2, we review the neccessary background on asymptotically flat static systems, weak inverse mean curvature flow, and outer-minimizing hypersurfaces, as well as notational conventions. In Section 3, we prove \thref{cmc_rigidity}. The key points of the proof are (a) the asymptotic convergence of weak inverse mean curvature flow completely determines the induced metric on each leaf of the umbilical foliation, and (b) the mean curvature of these leaves is uniquely determined by the induced metric and mean curvature of the boundary. The latter fact arises from the de-coupling of the evolution equation for mean curvature from the other evolving quantities. Because of this, well-posedness of the initial value problem for quasi-linear parabolic equations guarantees uniqueness. 

In Section 4, we prove \thref{level_set_rig}. We achieve this by scaling the metric $g$ by a conformal factor involving the potential function $V$, an approach with a rich history in the study of static manifolds-- see \cite{multiple_static_black_hole}, \cite{photon_sphere_positive_mass}, \cite{level_set_static}, \cite{photon_sphere_spinorial}, and \cite{structure_space_solutions}. Here, we consider the metric $g_{-} = V^{\frac{4}{n-2}} g$, which is also an asymptotically flat static metric on $M^{n}$. We obtain inequality \eqref{static_level_set} by applying inequality \eqref{inequality} with respect to this conformally modified static system. This is possible thanks in part to the elliptic maximum principle. In Section 5, we prove \thref{bartnik} on the uniqueness of static extensions for constant, Schwarzschild stable Bartnik data. The stability assumption allows us to show via local arguments that the boundary $\Sigma$ is totally umbilical and equipotential within the extension $(M^{n},g)$. In turn, we apply the level-set inequality \thref{level_set_rig} to show that the Minkowski inequality must be saturated on $(M^{n},g)$. We prove \thref{photon_surface} and \thref{equi_cmc} with essentially the same approach. In Section 6, we prove \thref{hawking_comparison}. We immediately obtain an estimate from below on the $L^{2}$ norm of $H_{\Sigma}$ using the conformal approach detailed in Section 4.

\vspace{0.5cm}

\noindent \textbf{Acknowledgements} Ye-Kai Wang is supported by Taiwan NSTC grant 109-2628-M-006-001-MY3. We would also like to thank Professor Mao-Pei Tsui of National Taiwan University for recommending the Christodoulou-Yau paper \cite{quasi_local_remarks}, and Professor Carla Cederbaum of T\"ubingen University for helpful email conversations on the topic of Schwarzschild rigidity.

    \section{Preliminaries}

We first clarify some notation throughout this paper. The ambient metric is denoted by $g$ or $g_{ij}$ and its scalar curvature by $R_{g}$. The Hessian $\nabla^{2}f$ with respect to $g$ for $f \in C^{\infty}(M)$ is denoted using bracket notation, i.e.

\begin{equation*}
    \nabla^{2} f (X,Y) = \langle \nabla_{X} \nabla f, Y \rangle \hspace{2cm} X,Y \in \Gamma(TM). 
\end{equation*}
The induced metric on a $C^{\infty}$ hypersurface $\Sigma^{n-1} \subset M^{n}$ is denoted by $\sigma$ or $\sigma_{ij}$ and its scalar curvature by $R_{\sigma}$. The standard sphere metric is denoted $\sigma^{\mathbb{S}^{n-1}}$, and its area is denoted by $w_{n-1}$. The second fundamental form of $\Sigma$,

\begin{equation*}
    h(X,Y)= -\langle \nu, \nabla_{X} Y \rangle = \langle \nabla_{X} \nu, Y \rangle \hspace{2cm} X,Y \in \Gamma(T\Sigma),
\end{equation*}
is taken with respect to the unit normal $\nu$ that points toward infinity. Occasionally, we will consider $\Sigma^{n-1}$ to be the immersed image of an abstract smooth manifold, and we will abuse notation by also denoting this manifold by $\Sigma$. 

\subsection{Asymptotically Flat Static Systems}
In general relativity, an isolated gravitating static system is modelled by a metric $g$ and potential $V$ which are asymptotic to the Euclidean metric and $1$, respectively.
\begin{definition}[Asymptotically Flat Static Systems] \thlabel{AF}
A Riemannian manifold $(M^{n},g)$ is \textbf{asymptotically flat} if there exists a compact set $K$ and a diffeomorphism $x=(x^{1},\dots,x^{n}): M^{n} \setminus K \rightarrow \mathbb{R}^{n} \setminus B_{1}(0)$ such that the metric $g$ satisfies the asymptotic expansion

\begin{equation} \label{g_expansion}
    g_{kl} = \delta_{kl} + \eta_{kl}, \hspace{1cm} \text{  with      } \eta_{kl} \in O_{2} (|x|^{-q}) \hspace{1cm} \text{  as  } |x| \rightarrow \infty,
\end{equation}
for some $q > \frac{n-2}{2}$ and for every $k,l \in \{ 1, \dots, n \}$. Here, for a function $f \in C^{\infty}(M \setminus K)$ we write
\begin{equation*}
    f \in O_{k}(|x|^{-q}) \hspace{1cm} \text{   if   } \hspace{1cm} \sum_{|J| \leq k} |x|^{q + |J|} |\partial^{J} f| = O(1) \hspace{1cm} \text{  as  } |x| \rightarrow \infty
\end{equation*}
in the coordinates $(x^{1}, \dots, x^{n})$.

Furthermore, if an asymptotically flat $(M^{n},g)$ admits a bounded, non-zero static potential $V$, then WLOG scaling $V$ so that

\begin{equation*}
    \lim_{|x| \rightarrow \infty} V(x) = 1
\end{equation*}
within the coordinates $(x^{1}, \dots, x^{n})$, we refer to the triple $(M^{n},g,V)$ as an \textbf{asymptotically flat static system}.
\end{definition}
\begin{remark}
Definitions of asymptotic flatness typically also include an integrability assumption on the scalar curvature $R_{g}$. However, taking the trace of the first equation in \eqref{static_equations}, we see that $R_{g} \equiv 0$ when $(M^{n},g)$ is static. Therefore, this assumption is redundant for our purposes.
\end{remark}
It is a non-trivial fact that a bounded, non-zero static potential $V$ is asymptotic to a constant within an asymptotically flat coordinate chart, c.f. for example Proposition B.4 in \cite{static_min}, and since the equations \eqref{static_equations} are linear in $V$ we may scale this constant to be $1$. Moreover, the asymptotic expansion of $V$ is given in \cite{static_min} Proposition B.4 as 

\begin{equation} \label{V_expansion}
    V(x) = 1 - m |x|^{2-n} + w(x), \hspace{1cm} \text{  with      } w \in o(|x|^{2-n}) \hspace{1cm} \text{  as  } |x| \rightarrow \infty,
\end{equation}
for a constant $m \in \mathbb{R}$. \cite{static_min} Proposition B.4 also shows within a chart $(x^{1},\dots,x^{n})$ satisfying \eqref{g_expansion} that

\begin{equation*}
    m = \frac{1}{2(n-1)w_{n-1}} \lim_{r \rightarrow \infty} \int_{|x|=r} \sum_{i,j=1}^{n} (\partial_{i} g_{ij} - \partial_{j} g_{ii}) \nu^{j} d\sigma. 
\end{equation*}
That is, the coefficient $m$ in \eqref{V_expansion} equals the \textit{ADM mass} of the manifold $(M^{n},g)$, which Bartnik shows in \cite{Bartnik1986TheMO} is a well-defined invariant of an asymptotically flat manifold in all dimensions. Because the ADM mass is encoded within the expansion of $V$, $m$ may also be shown to equal a boundary integral. This is observed by taking the limit of the normal derivative of $V$ over coordinate spheres at infinity and applying \eqref{V_expansion}:



\begin{equation*}
    \lim_{r \rightarrow \infty} \int_{|x|=r} \frac{\partial V}{\partial \nu} d\sigma = \lim_{r \rightarrow \infty} \int_{|x|=r} (n-2) m |x|^{1-n} d\sigma = (n-2) w_{n-1} m.
\end{equation*}
Since $V$ is harmonic with respect to $g$, the above integral is constant within a homology class of $M^{n}$. From this, we recover the boundary integral formula for the ADM mass
\begin{equation} \label{mass_formula}
    m = \frac{1}{(n-2)w_{n-1}}\int_{\partial M} \frac{\partial V}{\partial \nu} d \sigma.
\end{equation}
\eqref{mass_formula} is known in the literature as the ``Smarr formula", see \cite{PhysRevLett.30.71}, and we will use it extensively in sections 4 and 5.

\subsection{Inverse Mean Curvature Flow} In recent years, Inverse Mean Curvature Flow (IMCF) has yielded several Minkowski-type inequalities, see \cite{minkowski_euclidean}, \cite{ads_minkowski}, and \cite{schwarzschild_minkowski}. A one-parameter family of $C^{\infty}$ oriented immersions $X: \Sigma^{n-1} \times [0,T) \rightarrow M^{n}$ from a closed manifold $\Sigma^{n-1}$ into a Riemannian manifold $(M^{n},g)$ solves IMCF with initial data $\Sigma_{0} \subset M^{n}$ if

\begin{eqnarray} \label{IMCF}
    \frac{\partial}{\partial t} X (y,t) &=& H^{-1} \nu(y,t), \hspace{1.5cm} y \in \Sigma^{n-1}, \\
    \Sigma_{0} &=& X_{0}(\Sigma), \nonumber
\end{eqnarray}
where $\nu$ is the outward normal and $H>0$ the mean curvature of $\Sigma_{t}=X_{t}(\Sigma)$. A $C^{\infty}$, $H>0$ initial immersion $X_{0}$ always admits a short-time solution to \eqref{IMCF}, c.f. \cite{Gerhard1999}, but in general finite-time singularities eventually occur. Huisken and Ilmanen developed a notion of weak solutions to IMCF in \cite{imcf_penrose} to flow past these singularities.

Their weak solutions involve a level-set formulation of \eqref{IMCF}: for a smooth domain $\Omega \subset (M^{n},g)$, consider a solution $u \in C^{\infty}(M^{n})$, $du \neq 0$, to the degenerate elliptic Dirichlet problem

\begin{eqnarray} \label{level_set}
    \text{div}\left( \frac{\nabla u}{|\nabla u|} \right) &=& |\nabla u| \hspace{1.2cm} \text{  in  } (M^{n} \setminus \Omega,g), \\
    u|_{\Sigma_{0}} &=& 0 \hspace{2cm} \text{   on   } \partial \Omega. \nonumber
\end{eqnarray}
 Direct verification shows that the level-set flow of $u$ solves \eqref{IMCF} with initial condition $\Sigma_{0}= \partial \Omega$. Thus, \eqref{level_set} re-formulates of the initial value problem \eqref{IMCF}, with the flow surfaces $\Sigma_{t}$ corresponding to the level sets $\{ u = t\}$ of $u$.

The weak solutions developed in \cite{imcf_penrose} are variational solutions of \eqref{level_set}, and the $C^{1,\alpha}$ hypersurfaces defined by

\begin{equation}
    \Sigma_{t} = \partial \{ u < t\} \subset M^{n},
\end{equation}
are the corresponding ``weak flow surfaces". In the case of equality in \thref{minkowski}, the weak solution on the static manifold $(M^{n},g)$ is a smooth solution to \eqref{level_set}, which allows us to ignore many of the delicate aspects of the weak solution theory here. There are, however, some important asymptotic properties for general weak solutions which will play an important role in the analysis of the smooth solution. For convenience of the reader, we state these results below.  
\begin{theorem}[Lemma 7.1 in \cite{imcf_penrose}, Asymptotic Properties of Weak IMCF] \thlabel{imcf_asymptotics}
Let $(M^{n},g)$ be an asymptotically flat Riemannian manifold, let $U= \mathbb{R}^{n} \setminus B_{\delta}(0)$ be an asymptotically flat end of $(M^{n},g)$ with asymptotically flat coordinates $x^{i}$, and let $u \in C^{0,1}_{\text{loc}} (M^{n})$ be the unique variational solution to IMCF with initial condition $\Sigma_{0}$ and compact level sets. Then the following properties hold.

\begin{enumerate} [label=(\alph*)]
    \item The blown-down flow surfaces

    \begin{equation}
        \widetilde{\Sigma}_{t} =  \{ e^{-\frac{t}{n-1}}x | x \in \Sigma_{t} \}, \hspace{2cm} x=(x^{1}, \dots, x^{n+1}) \in U,
    \end{equation}

    converge in $C^{1,\alpha}$ to $\partial B_{r_{0}}(0)$ for $r_{0}$ defined before.
    \item There is a uniform constant $C$ so that the function $u$ satisfies $|\nabla u(x)| \leq C|x|^{-1}$ where differentiable and for $|x|$ large enough. In particular, $H_{\Sigma_{t}} \leq C e^{-\frac{t}{n-1}}$ when $u$ solves \eqref{level_set} classically.
\end{enumerate}

\begin{remark}
The decay conditions required by Huisken-Ilmanen for \thref{imcf_asymptotics} are actually more mild than stated above, and their definition of asymptotic flatness in \cite{imcf_penrose} requires stronger conditions than in \thref{AF}. However, these distinctions are not important to any of the proofs.
\end{remark}
\end{theorem}

\subsection{Outer-Minimizing Sets}
An important notion in geometry and relativity which is intimately related to IMCF is that of an outer-minimizing set.  

\begin{definition} \thlabel{outer_min}
Let $(M^{n},g)$ be a Riemannian manifold with smooth, compact boundary $\partial M^{n} =\Sigma^{n-1}$. We say that $\Sigma$ is \textbf{outer-minimizing} if for any $C^{1}$ domain $\Omega \subset (M^{n},g)$ that is bounded by $\Sigma$ we have $|\Sigma| \leq |\partial \Omega \setminus \Sigma|$, and \textbf{strictly outer-minimizing} if this inequality is strict.
\end{definition}
By the second variation of area, an outer-minimizing boundary must have mean curvature $H \geq 0$. Conversely, an $H \geq 0$ hypersurface in $(M^{n},g)$ is outer-minimizing if it a level set of a suitable function on the ambient space.

\begin{proposition}[Foliations by Mean-Convex Level Sets] \thlabel{outer_min_level}
Let $(M^{n},g)$ be a Riemannian manifold. Consider $u \in C^{2}(M)$ with $d u \neq 0$ and compact level sets $\Sigma_{t} = \{ u = t \}$. If $H_{\Sigma_{t}} \geq 0$ on every $\Sigma_{t}= \{ u = t \} \subset M^{n}$, then each $\Sigma_{t} = \partial (M^{n} \cap \{ u > t \})$ is outer-minimizing in $(M^{n} \cap \{ u > t\},g)$ (resp. strictly outer-minimizing if $H_{\Sigma_{t}} >0$).
\end{proposition}

\begin{proof}
Given that $M^{n}$ is foliated by the mean-convex hypersurfaces $\Sigma_{t}$, we have

\begin{equation} \label{level_set_H}
    \text{div}\left( \frac{\nabla u}{|\nabla u|} \right) \geq 0 \hspace{0.5cm} \text{in} \hspace{0.3cm} (M^{n},g)
\end{equation}
from the mean curvature formula for a level set. For each $t \in \text{Im} (u)$, let $\Omega \subset M^{n} \cap \{ u > t\}$ by a $C^{1}$ domain bounded by $\Sigma_{t}$. From the divergence theorem

\begin{eqnarray*}
    0 &\leq & \int_{\Omega} \text{div} \left(\frac{\nabla u}{|\nabla u|} \right) d \sigma =
    \int_{\partial \Omega \setminus \Sigma_{t}} \langle \frac{\nabla u}{|\nabla u|}, \nu \rangle d \sigma - \int_{\Sigma_{t}} \langle \frac{\nabla u}{|\nabla u|}, \nu \rangle d \sigma, \\
    &\leq& |\partial \Omega \setminus \Sigma_{t}| - |\Sigma_{t}|. 
\end{eqnarray*}
where we have used Cauchy-Schwarz on $\partial \Omega \setminus \Sigma_{t}$ and $\nu= \frac{\nabla u}{|\nabla u|}$ on $\Sigma_{t}$. If $H_{\Sigma_{t}} > 0$ for each $t$, then inequality \eqref{level_set_H} is strict.
\end{proof}
Foliations by IMCF are strictly outer-minimizing since they solve \eqref{level_set}. We presented \thref{outer_min_level} in a more general form because we are also interested in the level sets of static potentials on asymptotically flat manifolds. Due to the asymptotic behavior of $V$, c.f. \cite{level_set_static}, the sets $\{ V= V_{0} \}$ are each smooth, $H>0$ hypersurfaces outside of a compact set.

\begin{corollary}
Let $(M^{n},g, V)$ be an asymptotically flat static system with $m >0 $. Choose $\widetilde{V_{0}} < 1$ such that

\begin{enumerate}
    \item $d V \neq 0$ on $ \{ V > \widetilde{V_{0}} \} \subset M^{n}$, AND
    \item $H_{\{ V=V_{0} \}} > 0$ when $V_{0} \in (\widetilde{V_{0}},1)$.  
\end{enumerate}
Then for $V_{0} \in (\widetilde{V}_{0},1)$ the level sets $\Sigma^{n-1} = \{ V= V_{0} \} \cong \mathbb{S}^{n-1}$ are smooth and strictly outer-minimizing in $\{ V > V_{0} \} \subset (M^{n},g)$. In particular, \thref{level_set_rig} applies. 
\end{corollary}
This corollary is not essential to our approach, but it may be useful elsewhere, e.g. for understanding the asymptotic behavior of the level sets of $V$.

\section{Umbilical Foliations by Inverse Mean Curvature Flow}    

The goal of this section will be to prove \thref{cmc_rigidity}. In \cite{schwarzschild_minkowski}, Wei characterizes equality when $(M^{n},g)$ is a region in the Schwarzschild manifold. However, his method of proof cannot be applied to an arbitrary static background, and so our approach here is quite different. \thref{minkowski} states that whenever equality is achieved in \eqref{inequality}, $\partial M= \Sigma$ (i.e. $\partial M$ is connected), and $M$ is foliated by a smooth, totally umbilical solution $\Sigma_{t}$ to IMCF.

The first important observation is that when the $\Sigma_{t}$ are umbilical, the induced metric $\sigma_{ij}(y,t)$ of $\Sigma_{t}$ evolves self-similarly under \eqref{IMCF}. This de-couples the evolution of the mean curvature $H(y,t)$ from all other evolving geometric quantities, implying that $g$ is uniquely determined by $H(y,0)$ and $\sigma_{ij}(y,0)$. We state the following theorem for the more general situation of a background space with constant scalar curvature, since it may be useful in understanding other geometric inequalities arising from IMCF.

\begin{theorem}[Umbilical Foliations by IMCF] \thlabel{imcf_uniqueness}
Let $(M^{n},g)$ be a Riemannian manifold with constant scalar curvature $R_{g}= \overline{R_{0}}$ and $C^{\infty}$, $H>0$ boundary $\partial M$. Let $\partial M= X_{0}(\Sigma)$ for a smooth manifold $\Sigma$ and embedding $X_{0}$, and let $X: \Sigma \times (0,T) \rightarrow (M^{n},g)$ be the solution to \eqref{IMCF} with initial data $X_{0}$. If $\Sigma_{t}=X_{t}(\Sigma)$ foliate $M^{n}$ and are umbilic, then the pull-back of $g$ under $X$ is

\begin{equation*}
    g_{ij}(y,t) = u(y,t)^{2} dt^{2} + e^{\frac{2t}{n-1}}\sigma_{ij}(y,0),
\end{equation*}
where $\sigma_{ij}(y,0)$ is the induced metric on $\partial M$ pulled back via $X_{0}$, and $u(y,t)> 0$ solves the initial value problem

\begin{eqnarray} \label{initial_value}
    \partial_{t} u (y,t) &=& e^{-2\frac{t}{n-1}} u^{2} \Delta_{\sigma_{0}} u + \frac{n}{2(n-1)}u \\
    & &+ \left(\frac{\overline{R_{0}}}{2} - \frac{R_{\sigma_{0}}(y)}{2}e^{-\frac{2t}{n-1}} \right) u^{3}, \nonumber \\
    u(y,0) &=& H^{-1}(y) \nonumber
\end{eqnarray}
on $\Sigma \times [0,T)$. Here $\Delta_{\sigma_{0}}$ and $R_{\sigma_{0}}$ are the Laplace-Beltrami operator and intrinsic scalar curvature function of the metric $\sigma_{ij}(y,0)$, and $H(y)$ is the mean curvature of $\partial M$ pulled back under $X_{0}$.

In particular, $g$ is uniquely determined by the induced metric and mean curvature of $\partial M$.
\end{theorem}

\begin{proof}
We first determine the induced metric $\sigma_{ij}(y,t)$ of $\Sigma_{t}$. Since $h_{ij}(y,t)= \frac{H}{n-1} \sigma_{ij}(y,t)$, the first variation formula for the metric, c.f. Theorem 3.2(i) in \cite{Gerhard1999}, becomes 
\begin{equation} \label{self-similar}
    \frac{\partial}{\partial t} \sigma_{ij}(y,t) = \frac{2}{H} h_{ij}(y,t) = \frac{2}{n-1} \sigma_{ij}(y,t).
\end{equation}
As a result, $\sigma_{ij}(y,t)=e^{\frac{2t}{n-1}} \sigma_{ij}(y,0)$. The evolution equation for the flow speed $H^{-1}$ in a general Riemannian manifold is given, e.g. in equation (1.3) of \cite{imcf_penrose}, as

\begin{equation} \label{flow_speed}
    \partial_{t} H^{-1} = \frac{1}{H^{2}} \Delta H^{-1} + \left(\frac{|h|^{2}}{H^{2}} + \frac{\text{Ric}(\nu,\nu)}{H^{2}} \right) H^{-1}.
\end{equation}
Since $\Sigma_{t}$ is umbilical with induced metric evolving by scaling,

\begin{eqnarray*}
    |h|^{2}               &=& \frac{1}{n-1}H^{2}, \\
    \Delta_{\sigma_{t}} f &=& e^{-\frac{2t}{n-1}} \Delta_{\sigma_{0}} f, \\
    \text{Ric}(\nu,\nu) &=& \frac{1}{2} \left( \frac{n-2}{n-1}H^{2} - R_{\sigma_{t}}(y,t) + R_{g} \right) \\
                        &=& \frac{1}{2} \left( \frac{n-2}{n-1}H^{2} - e^{-\frac{2t}{n-1}} R_{\sigma_{0}}(y) + \overline{R}_{0} \right)
\end{eqnarray*}
where the last two lines come from the twice-contracted Gauss equation. This reduces \eqref{flow_speed} to

\begin{eqnarray*}
\partial_{t} H^{-1} (y,t) &=& e^{-2\frac{t}{n-1}} H^{-2} \Delta_{\sigma_{0}} H^{-1} + \frac{n}{2(n-1)}H^{-1} \\
    & &+ \left(\frac{\overline{R_{0}}}{2} - \frac{R_{\sigma_{0}}(y)}{2}e^{-\frac{2t}{n-1}} \right) H^{-3}, \nonumber
\end{eqnarray*}
and so altogether, $u=H^{-1}$ solves the initial value problem \eqref{initial_value}.
The right-hand side of the evolution equation \eqref{initial_value} equals $L(u)$ for a second-order differential operator $L$ which in coordinates has the form

\begin{equation*}
 L(u) = Q^{ij}(y, t, u) \partial^{2}_{ij} u + b(y,t,u).    
\end{equation*}
Here, $Q^{ij}$ is a symmetric matrix with uniformly positive eigenvalues over compact time intervals. Therefore, the problem \eqref{initial_value} is of quasi-linear parabolic type, and so its solution is uniquely determined by the initial data $u(y,0)$-- see Theorem A.3.1 in \cite{mcf_lecture_notes}, see also Theorem 1.1 in \cite{parabolic_uniqueness}. 
\end{proof}

The second important observation is that the asymptotics of IMCF in an asymptotically flat manifold completely determine the topology and intrinsic geometry of these $\Sigma_{t}$. It is clear that the $\Sigma_{t}$ which foliate $(M^{n},g)$ are neccessarily diffeomorphic to one another, as they are the level sets of a smooth, non-degenerate function $u: M^{n} \rightarrow \mathbb{R}$, and so $\Sigma_{t} \cong \mathbb{S}^{n-1}$ for each $t \in [0,\infty)$ according to \thref{imcf_asymptotics}. Those asymptotics also allow us to fully determine $\sigma_{ij}(y,t)$ for each $t \in [0,\infty)$. Specifically, because $\Sigma_{t}$ must become $C^{1,\alpha}$-close to a coordinate sphere as $t \rightarrow \infty$, each $\Sigma_{t}$ is intrinsically round whenever the equality is achieved.

\begin{theorem} \thlabel{round}
Let $(M^{n},g)$ be an asymptotically flat Riemannian manifold with boundary $\partial M$. Suppose $(M^{n},g)$ is foliated by a smooth solution $\{\Sigma_{t}\}_{0 \leq t < +\infty}$ of IMCF with initial data $\Sigma_{0}=\partial M$. If each $\Sigma_{t}$ is umbilical, then the induced metric $\sigma_{ij}(y,t)$ on $\Sigma_{t}$ is

\begin{equation*}
   \sigma_{ij} (y,t)= r_{0}^{2}e^{\frac{2t}{n-1}} \sigma_{ij}^{\mathbb{S}^{n-1}},
\end{equation*}
 where $r_{0}= (\frac{|\Sigma_{0}|}{w_{n-1}})^{\frac{1}{n-1}}$ and $\sigma^{\mathbb{S}^{n-1}}_{ij}$ is the standard sphere metric.
\end{theorem}
\begin{remark}
Coordinate spheres are the only closed, embedded hypersurfaces in the Schwarzschild manifold which are both intrinsically round ($\sigma_{ij}= r_{0}^{2} \sigma_{ij}^{\mathbb{S}^{n-1}}$) and extrinsically round ($h_{ij}= \frac{H}{n-1} \sigma_{ij}$). Therefore, \thref{round} gives a direct proof of rigidity in Wei's Theorem 1.1 from \cite{schwarzschild_minkowski}. 
\end{remark}

\begin{proof}
Let $U= \mathbb{R}^{n} \setminus B_{\delta}(0)$ be an asymptotically flat end of $(M^{n},g)$ with coordinates $x = (x^{1}, \dots, x^{n})$, and for each $t >0$ define blow-down objects on $U$ by

\begin{equation*}
   \widetilde{X}_{t} (y)= e^{-\frac{t}{n-1}} X_{t} (y), \hspace{2cm}  g^{(t)}_{kl} (x) = e^{-\frac{2t}{n-1}} g_{kl}(e^{\frac{t}{n-1}} x).
\end{equation*}
for $y \in \mathbb{S}^{n-1}$ and immersions $X_{t}: \mathbb{S}^{n-1} \rightarrow U$ solving IMCF in the  gauge \eqref{IMCF}. Note that $\widetilde{X}_{t}$ is an embedding if $X_{0}$ is, since here \eqref{IMCF} is the level-set flow of the function $u$ of \eqref{level_set}. Note also by asymptotic flatness with decay order $q > \frac{n-2}{2} \geq \frac{1}{2}$ that

\begin{equation} \label{af_decay}
    \sup_{x \in U}  (|g^{(t)}_{kl} (x) - \delta_{kl}| + e^{\frac{t}{n-1}} |\partial_{m} g^{(t)}_{kl}(x)| +  e^{\frac{2t}{n-1}} |\partial_{mr} g^{(t)}_{kl}|) \leq C e^{-\frac{t}{2(n-1)}} \hspace{0.6cm} m, r=1, \dots, n.  
\end{equation}
Since $h_{ij}(y,t)= \frac{H}{n-1} \sigma_{ij}(y,t)$, we know $\sigma_{ij}(y,t) = e^{\frac{2t}{n-1}} \sigma_{ij}(y,0)$ in view of the variation formula \eqref{self-similar}. Then the pull-back of $g_{kl}^{(t)}$ under the embedding $\widetilde{X}_{t}$ is

\begin{equation} \label{fixed_metric}
   \widetilde{\sigma}_{ij} (y,t) = \sigma_{ij}(y,0).
\end{equation}
On the other hand, the extrinsic curvature $\widetilde{h}_{ij}(y,t)$ of $\widetilde{\Sigma}_{t}$ satisfies

\begin{eqnarray} \label{total_curvature}
    |\widetilde{h}(y,t)|_{\widetilde{\sigma}(y,t)} &=& |\frac{\widetilde{H}}{n-1} \widetilde{\sigma}(y,t)|_{\widetilde{\sigma}(y,t)} \\
    &=& \frac{1}{\sqrt{n-1}} \widetilde{H}(y,t) = \frac{1}{\sqrt{n-1}} e^{\frac{t}{n-1}} H(y,t) \leq C. \nonumber
\end{eqnarray}
for some uniform constant $C$. Here, we have used the estimate $H=|\nabla u| \leq C e^{-\frac{t}{n-1}}$ on mean curvature from property (b) of \thref{imcf_asymptotics}. 

Denoting the abstract initial metric by $\sigma_{0}$, $\widetilde{X}_{t}: (\mathbb{S}^{n-1}, \sigma_{0}) \rightarrow (U, g^{(t)})$ are each isometric embeddings with uniformly bounded total curvature $|\widetilde{h}|_{\sigma_{0}}$. In particular, the Hessians for the coordinate functions $\widetilde{X}^{i}_{t} \in C^{\infty} (\mathbb{S}^{n-1})$, $i= 1, \dots, n$ satisfy
\begin{eqnarray*}
    |\nabla^{2}_{\sigma_{0}} \widetilde{X}^{i}_{t}|_{\sigma_{0}} &=& 
    |\nabla^{2}_{g^{(t)}} x^{i} - \nu(x^{i}) \widetilde{h}|_{\sigma_{0}} \\
    &\leq& |\Gamma^{i}_{kl}(t)| + |\nu(x^{i})| |\widetilde{h}|_{\sigma_{0}} \leq C,
\end{eqnarray*}
where $\Gamma^{i}_{kl}(t)$ are the Christoffel symbols of $g^{(t)}_{kl}$ in the coordinates $(x^{1}, \dots, x^{n})$. The upper bounds on $|\Gamma^{i}_{kl}(t)|$, $|\widetilde{h}|_{\sigma_{0}}$, and $|\nu(x^{i})|$ follow from \eqref{af_decay}, \eqref{total_curvature}, and property (a) of \thref{imcf_asymptotics}, respectively. Thus a local coordinate system $(y^{1},\dots,y^{n-1})$ of $\mathbb{S}^{n-1}$, we have for a fixed constant $C$ that

\begin{equation*}
    |\frac{\partial}{\partial y_{j}} \widetilde{X}^{i}_{t}(w) - \frac{\partial}{\partial y_{j}} \widetilde{X}^{i}_{t}(z)| \leq C d_{\sigma_{0}} (w,z) \hspace{1cm} w,z \in \mathbb{S}^{n-1}, \hspace{0.5cm} j=1, \dots, n-1 \hspace{0.5cm} i=1,\dots,n.
\end{equation*}
By the compact containment $C^{1,1}(\mathbb{S}^{n-1},\sigma_{0}) \subset \subset C^{1,\alpha}(\mathbb{S}^{n-1},\sigma_{0})$, see \cite{sobolev}, we may choose a subsequence $\widetilde{X}^{i}_{t_{k}}$ of $\widetilde{X}^{i}_{t}$ so that for every $i=1,\dots,n$
\begin{equation*}
    \widetilde{X}^{i}_{t_{k}} \rightarrow X^{i}_{\infty} \hspace{1cm} \text{  in  } C^{1,\alpha}(\mathbb{S}^{n-1},\sigma_{0}) \hspace{0.5cm} \alpha \in (0,1).
\end{equation*}
We now consider the map into $U$ with coordinate functions $X^{i}_{\infty}$. Since

\begin{equation*}
    \sigma_{ij}(y,0) = g^{(t)}_{kl} \partial_{i} \widetilde{X}_{t}^{k} \partial_{j} \widetilde{X}_{t}^{l} =  \delta_{kl} \partial_{i} \widetilde{X}_{t}^{k} \partial_{j} \widetilde{X}_{t}^{l} + \eta_{ij} \hspace{1cm} |\eta_{ij}|= |(g^{(t)}_{kl} - \delta_{kl}) \partial_{i} \widetilde{X}_{t}^{k} \partial_{j} \tilde{X}_{t}^{l}| \leq C e^{-\frac{t}{2(n-1)}}
\end{equation*}
by \eqref{af_decay}, the pull-back $\sigma_{ij}(y,\infty)= \delta_{kl} \partial_{i} X_{\infty}^{k} \partial_{j} X_{\infty}^{l}$ of $\delta$ under the map $X_{\infty}$ equals $\sigma_{ij}(y,0)$. That is

\begin{equation*}
    X_{\infty}: (\mathbb{S}^{n-1},\sigma_{0}) \rightarrow (U,\delta)
\end{equation*}
is a $C^{1}$ isometric embedding. Now, property (a) in \thref{imcf_asymptotics} means that $\lim_{k \rightarrow \infty} |\widetilde{X}_{t_{k}}| = (\frac{|\Sigma_{0}|}{w_{n-1}})^{\frac{1}{n-1}}  = r_{0}$, and so

\begin{equation*}
    X_{\infty}(\mathbb{S}^{n-1}) = \partial B_{r_{0}}(0) \subset U.
\end{equation*}
Any isometry $R$ of $\partial B_{r_{0}} \subset (U, \delta)$ defines a $C^{1}$ isometry $\widetilde{R}: (\mathbb{S}^{n-1}, \sigma_{0}) \rightarrow (\mathbb{S}^{n-1}, \sigma_{0})$ via the conjugation

\begin{equation*}
    \widetilde{R}= X^{-1}_{\infty} \circ R \circ X_{\infty}.
\end{equation*}
By conjugation with the group of rotations of $\partial B_{r}(0)$, we see that $\text{Isom}(\sigma_{0})$ acts transitively on $\mathbb{S}^{n-1}$, and that for each $y \in \mathbb{S}^{n-1}$ the stabilizer $\text{Stab}_{y} \subset \text{Isom}(\sigma_{0})$ acts transitively on the set $\{ v \in T_{y} \mathbb{S}^{n-1} | \sigma_{0}(v,v) \} =1$. Taken together, these imply that the sectional curvatures of $\sigma_{0}$ are identically equal and constant.
\end{proof}

When $\partial M$ has constant mean curvature and constant intrinsic scalar curvature, the initial value problem \eqref{initial_value} may be solved by considering the corresponding ODE for $u$. This implies that the $\Sigma_{t}$ in \thref{imcf_uniqueness} are CMC and that $g$ is a warped product metric. We obtain from this the following rigidity theorem for the Schwarzschild manifold, also stated under the more general assumption that $R_{g}=0$.

\begin{theorem}[Uniqueness of Umbilical Foliations in Vacuum] \thlabel{umbilic_rig}
Let $(M^{n},g)$ be an asymptotically flat Riemannian manifold with scalar curvature $R_{g}=0$ and CMC boundary $\partial M= \Sigma$. Suppose that $M^{n}$ is foliated by a solution $\{ \Sigma_{t} \}_{0 < t < \infty}$ to IMCF such that each $\Sigma_{t}$ is totally umbilical.

Then $(M^{n},g)$ is isometric to the exterior of the coordinate sphere $\{ r= r_{0} \}$ in the Schwarzschild manifold of mass $m_{0}$ for the constants $r_{0}$, $m_{0}$ of $\Sigma$ given in \thref{stability}.
\end{theorem}

\begin{proof}
By \thref{round}, $\sigma_{ij}(y,0) = r_{0}^{2} \sigma^{\mathbb{S}^{n-1}}_{ij}$, and so

\begin{equation*}
    g_{ij}(y,t) = u(y,t)^{2} dt^{2} + r_{0}^{2}e^{\frac{2t}{n-1}} \sigma_{ij}^{\mathbb{S}^{n-1}}
\end{equation*}
for $u$ satisfying \eqref{initial_value} with $R_{g}=0$ and $R_{\sigma_{0}} = (n-1)(n-2) r_{0}^{-2}$. Since $u(y,0)= H_{0}^{-1}$ for the constant $H_{0}$, the function $u$ is uniquely determined by solution of the corresponding ODE problem
\begin{eqnarray} \label{IVP}
    \frac{d}{dt} u(t)&=& \frac{n}{2(n-1)} u(t) - \frac{(n-1)(n-2)r_{0}^{-2}}{2} e^{-\frac{2t}{n-1}} u(t)^{3} \\
   u(0) &=& H_{0}^{-1} \nonumber
\end{eqnarray}
in view of \thref{imcf_uniqueness}. Let us make the transformation

\begin{eqnarray*}
    r(t)&=& r_{0} e^{\frac{t}{n-1}}, \\
    U(r) &=& \frac{1}{n-1} \frac{r}{u(r)},
\end{eqnarray*}
so that \eqref{IVP} becomes

\begin{eqnarray*}
    \frac{d}{dr} U &=&  \frac{1}{n-1} u(r)^{-1} - \frac{r}{(n-1) u(r)^{2}} \frac{d}{dr} u(r) \\
    &=& \frac{1}{n-1} u(r)^{-1} - \frac{1}{u(r)^{2}} \left(\frac{n}{2(n-1)} u(r) - \frac{(n-1)(n-2)}{2r^{2}} u(r)^{3} \right) \\
    &=& \frac{1}{r} U - \frac{(n-1)^{2}}{r^{2}}U^{2} \left( \frac{nr}{2 (n-1)^{2}} U^{-1} - \frac{(n-2)r}{2(n-1)^{2}} U^{-3} \right) \\
    &=& \frac{(n-2)}{2} \frac{U^{-1} - U}{r}, \\
    U(r_{0}) &=& \frac{1}{n-1} H_{0}r_{0} = U_{0}.
\end{eqnarray*}
The solution is

\begin{eqnarray}
    U(r)&=&\sqrt{1- 2 \frac{m_{0}}{r^{n-2}}}, \nonumber \\
    m_{0} &=& \frac{1}{2} r_{0}^{n-2} \left( 1 - U_{0}^{2} \right), \nonumber \\
          &=& \frac{1}{2} r_{0}^{n-2} \left( 1- \frac{1}{(n-1)^{2}} H_{0}^{2}r_{0}^{2} \right), \nonumber
\end{eqnarray}
and computing the mass $m$ of the metric

\begin{equation*}
    g_{ij}= \frac{1}{U(r)^{2}} dr^{2} + r^{2} \sigma^{\mathbb{S}^{n-1}}_{ij}
\end{equation*}
yields $m=m_{0}$ (also note that $U_{0} >0$ implies $r_{0} > (2m)^{\frac{1}{n-2}}$). 
\end{proof}

\thref{umbilic_rig} covers the case $H_{\Sigma} = H_{0}$ in \thref{cmc_rigidity}. For the case $V|_{\Sigma} = V_{0}$, we use the Codazzi equation and the Gauss equation for a hypersurface $\Sigma^{n-1}$ in a Riemannian manifold $(M^{n},g)$, which are respectively given by

\begin{eqnarray}
    \text{div}h (X) - X(H) &=& \text{Ric}(\nu, X), \label{codazzi} \\
     \text{Ric}(\nu,\nu) &=& \frac{1}{2} \left( \frac{n-2}{n-1}H^{2} - |\mathring{h}|^{2} - R_{\sigma} + R_{g} \right). \label{gauss}
\end{eqnarray}
Here, $X \in T_{y} \Sigma$ is a tangent vector and $\text{div} h (X)= \nabla_{e_{i}} h(e_{i},X)$ is the divergence one-form of the $2$-tensor $h$. Knowing that $\sigma_{ij}= r_{0}^{2} \sigma^{\mathbb{S}^{n-1}}_{ij}$ on $\Sigma$ whenever the equality is achieved, these equations allow us to reduce everything to the CMC case.
\begin{proposition} \thlabel{cmc}
Let $(M^{n},g)$ be a static Riemannian manifold with static potential $V$. Suppose $\Sigma^{n-1} \subset M^{n}$ is an immersed hypersurface such that

\begin{eqnarray*}
    V|_{\Sigma} &=& V_{0}, \\
    R_{\sigma} &=& R_{0}, \textit{    and} \\
    h_{ij} &=& \frac{H}{n-1} \sigma_{ij},
\end{eqnarray*}
for constants $V_{0},R_{0}$. Then $\Sigma$ is CMC.
\end{proposition}

\begin{proof}

We begin by considering the Codazzi equation on $\Sigma$. Since $h$ is pure trace, the left-hand side of \eqref{codazzi} is

\begin{equation*}
    \text{div}h (X) - X(H) = -\frac{n-2}{n-1} X(H).
\end{equation*}
On the other hand, the static equation transforms the right-hand side as

\begin{equation*}
    \text{Ric}(X,\nu) = \frac{1}{V} \langle \nabla_{X} \nabla V, \nu \rangle = \frac{1}{V_{0}} X (\frac{\partial V}{\partial \nu} ),
\end{equation*}
where we have used that $V|_{\Sigma}$ is constant. Altogether,

\begin{equation*}
   -\frac{n-2}{n-1} X(H) = \frac{1}{V_{0}} X( \frac{\partial V}{\partial \nu}).
\end{equation*}
and so
\begin{equation} \label{H_nu(V)}
    \frac{\partial V}{\partial \nu}(y) = - \frac{n-2}{n-1} V_{0}H(y) + \beta
\end{equation}
for some fixed $\beta \in \mathbb{R}$. Finally, we combine the Gauss equation \eqref{gauss} with the static equation on surfaces
\begin{equation} \label{surface_static}
    \Delta_{\sigma} V + \frac{\partial V}{\partial \nu} H = -V \text{Ric} (\nu, \nu),
\end{equation}
and in view of \eqref{H_nu(V)} we obtain

\begin{equation*}
    - \frac{n-2}{n-1} V_{0} H^{2} + \beta H = \frac{1}{2}\left( V_{0} R_{0} - \frac{n-2}{n-1} V_{0} H^{2} \right). 
\end{equation*}
Differentiating gives 
 
\begin{equation} \label{X(H)}
 \left( -\frac{n-2}{n-1} V_{0} H(y) + \beta \right) X(H) = 0
\end{equation}
for $x \in \Sigma$, $X \in T_{x} \Sigma$. Suppose that $X(H) \neq 0$. Then $H(y)= \frac{n-1}{n-2} \beta V_{0}^{-1}$. Extending $X$ to a smooth vector field in a neighborhood $U_{x} \subset \Sigma$, we consider the integral curve $\alpha$ of $X$ with $\alpha(0)=y$. For $t \in (0,\epsilon)$, we have $\frac{d}{dt} H(\alpha(t)) \neq 0$ and $H(\alpha(t)) \neq \frac{n-1}{n-2} \beta V_{0}^{-1}$. This contradicts \eqref{X(H)}. Therefore, $\nabla_{\Sigma} H =0$ identically on $\Sigma$.
\end{proof}
This proves \thref{cmc_rigidity}. We remark that full rigidity of \eqref{inequality} is more difficult to understand (we tried). Nevertheless, this could provide a fruitful research direction in the future.





\section{A Conformal Approach}

In this section, we prove \thref{level_set_rig}. To introduce our approach, we recall the isotropic coordinate chart $U= \mathbb{R}^{n} \setminus B_{(\frac{m}{2})^{\frac{1}{n-2}}}(0)$ of the Schwarzschild manifold $(M^{n},g_{m})$, $m>0$, wherein the metric and potential have the form

\begin{eqnarray*}
    g_{m}(x) &=& \left( 1 + \frac{m}{2} |x|^{2 - n} \right)^{\frac{4}{n-2}} \delta, \\
    V_{m}(x) &=& \frac{1 - \frac{m}{2} |x|^{2 - n}}{1 + \frac{m}{2} |x|^{2 - n}},
\end{eqnarray*}
with $\delta$ being the Euclidean metric on $U$. These coordinates make clear that the Schwarzschild metrics of mass $m$ and $-m$ on $M^{n}$ are related by the conformal transformation

\begin{equation*}
    g_{-m} = V_{m}^{\frac{4}{n-2}} g_{m}.
\end{equation*}
Indeed, this type of conformal transformation can be made for \textit{any} static metric, a fact used for example by Anderson in \cite{structure_space_solutions}, Section 1.

\begin{proposition} 
Let $(M^{n},g,V)$ be an asymptotically flat static system with (possibly disconnected) boundary $\partial M$. If $V_{\partial M} >0$, then the conformal metric and corresponding function

\begin{eqnarray} \label{conformal_metric}
    g_{-} &=& V^{\frac{4}{n-2}} g, \\
    V_{-} &=& V^{-1},
\end{eqnarray}
also form an asymptotically flat static system on $M^{n}$.
\end{proposition}

\begin{remark}
$H_{\partial M} > 0$ is a sufficient condition to ensure $V_{\partial M} >0$. Indeed, if $H_{\partial M} >0$ and $V(x) = 0 = \min_{\overline{M^{n}}} V$ for $x \in \partial M$, then $\Delta_{\partial M} V(x) \geq 0$ by the maximum principle and $\frac{\partial V}{\partial \nu}(x) > 0$ by the Hopf Lemma. Therefore,

\begin{equation*}
    \Delta_{\partial M} V (x) + H(x) \frac{\partial V}{\partial \nu} (x) > 0,
\end{equation*}
which contradicts the static equation \eqref{surface_static}.
\end{remark}

\begin{proof}
Since $\inf_{M} V > 0$, conformal scalings of $g$ by a power of $V$ produce a regular metric. It is easier to see that \eqref{conformal_metric} is static if we first consider the conformal scaling

\begin{equation*}
    g_{0} = V^{\frac{2}{n-2}} g.
\end{equation*}
For any open domain $\Omega \subset M$, the function $U= \ln (V)$ satisfies

\begin{eqnarray} \label{conformal_harmonic}
    \int_{\Omega} \Delta_{g_{0}} U d g_{0} &=& \int_{\partial \Omega} \frac{\partial U}{\partial \nu_{0}} d \sigma_{0} = \int_{\partial \Omega} V^{-\frac{1}{n-2}} \frac{\partial U}{\partial \nu} V^{\frac{n-1}{n-2}} d\sigma \\
    &=& \int_{\partial \Omega} \frac{\partial V}{\partial \nu} d \sigma. \nonumber
\end{eqnarray}
On the other hand, the transformation formula for Ricci curvature, c.f. page 59 in \cite{einstein_manifolds}, is

\begin{eqnarray} \label{conformal_ricci}
 \text{Ric}_{g_{o}} &=& \text{Ric}_{g} - \nabla_{g}^{2} U  + \frac{1}{n-2} d U \otimes d U - \frac{1}{n-2} (\Delta_{g} U +  |\nabla_{g} U|^{2}) g \\
 &=& \text{Ric}_{g} - V^{-1} \nabla_{g}^{2} V + V^{-2} d V \otimes d V + \frac{1}{n-2} d U \otimes d U - \frac{V^{-1}}{(n-2)} \Delta_{g} V g \nonumber \\
 &=& \text{Ric}_{g} - V^{-1} \nabla_{g}^{2} V - \frac{V^{-1}}{(n-2)} \Delta_{g} V g + \frac{n-1}{n-2} d U \otimes d U \nonumber
\end{eqnarray}
From \eqref{conformal_harmonic} and \eqref{conformal_ricci}, we see that $g$ and $V$ solve the static equations if and only if $g_{0}$ and $U$ solve the system

\begin{eqnarray*}
    \text{Ric}_{g_{0}} &=& \frac{n-1}{n-2} d U \otimes d U, \label{conformal_system} \\
    \Delta_{g_{0}} U &=& 0. \nonumber
\end{eqnarray*}
In fact, these equations are invariant under the change $U \rightarrow - U$, and so the metrics

\begin{eqnarray*}
    g_{\pm} = e^{\mp \frac{2}{n-2} U } g_{0}
\end{eqnarray*}
are both static with respective potential functions $V_{\pm}=e^{ \pm U}$ (we choose notation so that $g_{+}=g$ and $g_{-}= V^{\frac{4}{n-2}} g$ for a reason that will become clear shortly).

To conclude, we address the asymptotics of the conformal system. The derivatives of $(g_{-})_{kl}$ may be expressed as derivatives of $g_{kl}$ and $V$, and so the expansions \eqref{g_expansion} and \eqref{V_expansion} imply

\begin{equation*}
    (g_{0})_{kl} = \delta_{kl} + \eta_{kl}, \hspace{1cm} \eta_{kl} \in o_{2}(|x|^{\frac{1}{2}(2-n)}).
\end{equation*}
Likewise, since $V(x) \rightarrow 1$ as $|x| \rightarrow \infty$, the potential $V_{-}=V^{-1}$ must expand on the same orders as $V$.
\end{proof}

The sign of the ADM mass of $(M^{n},g)$ changes under the conformal scaling \eqref{conformal_metric}. However, \thref{minkowski} applies regardless of the sign of $m$, c.f. Remark 1.5 in \cite{static_minkowski}, and so we may apply this theorem to $(M, g_{-})$.

\begin{theorem} \thlabel{conformal_rig}
Let $(M^{n},g,V)$, $3 \leq n \leq 7$, be an asymptotically flat static system with connected, outer-minimizing boundary $\partial M = \Sigma$. If $\Sigma$ is outer-minimizing with respect to the metric $g_{-}$ defined in \eqref{conformal_metric}, then we have the inequality

\begin{equation} \label{conformal_inequality}
    \frac{1}{(n-1)w_{n-1}}\int_{\Sigma} VH d \sigma \geq \left(\frac{\int_{\Sigma} V^{2 \frac{n-1}{n-2}} d\sigma}{w_{n-1}} \right)^{\frac{n-2}{n-1}}.
\end{equation}
\end{theorem}

\begin{proof}
Once again using the integral formula \eqref{mass_formula}, the ADM mass $m_{-}$ of the metric $g_{-}$ is

\begin{eqnarray} \label{m_{+}}
    m_{-} &=& \frac{1}{(n-2) w_{n-1}} \int_{\Sigma} \frac{\partial V_{-}}{\partial \nu_{-}} d \sigma_{-} = - \frac{1}{(n-2) w_{n-1}} \int_{\Sigma} \frac{\partial V}{\partial \nu} d \sigma \\ 
    &=& -m, \nonumber
\end{eqnarray}
where $m$ is the ADM mass of $g$. The transformation formula for mean curvature, c.f. for example equation 4.3 in \cite{photon_sphere_spinorial} or equation 7 in \cite{pmt_less_than_8}, gives

\begin{equation*}
    H_{-} = V^{-\frac{n}{n-2}} ( 2 \frac{n-1}{n-2} \frac{\partial V}{\partial \nu} + HV). 
\end{equation*}
In particular,

\begin{eqnarray} \label{HV_conformal}
    \frac{1}{(n-1) w_{n-1}}\int_{\Sigma} V_{-} H_{-} d\sigma_{-} &=& \frac{1}{(n-1) w_{n-1}} \int_{\Sigma} V^{-\frac{2}{n-2}} \left( 2 \frac{n-1}{n-2} \frac{\partial V}{\partial \nu} + VH \right) V^{\frac{2}{n-2}} d\sigma \nonumber \\
    &=& \frac{2}{(n-2) w_{n-1}} \int_{\Sigma} \frac{\partial V}{\partial \nu} d\sigma + \frac{1}{(n-1)w_{n-1}} \int_{\Sigma} VH d \sigma \\
    &=& -2m_{-} + \frac{1}{(n-1)w_{n-1}} \int_{\Sigma} VH d \sigma \nonumber
\end{eqnarray}
Now, $g_{-}$ is asymptotically flat and static with the appropriately scaled potential $V_{-}$. Since $\Sigma$ is outer-minimizing with respect to $g_{-}$, the Minkowski inequality

\begin{equation*}
    \frac{1}{(n-1) w_{n-1}} \int_{\Sigma} V_{-} H_{-} d \sigma_{-} + 2m_{-} \geq \left( \frac{|\Sigma|_{\sigma_{-}}}{w_{n-1}} \right)^{\frac{n-2}{n-1}}
\end{equation*}
holds on $\Sigma$. Rewriting in terms of the original geometric quantities using \eqref{m_{+}} and \eqref{HV_conformal} yields \eqref{conformal_inequality}.
\end{proof}
\begin{proof}[Proof of \thref{level_set_rig}]
We would like to apply inequality \eqref{conformal_inequality}. We assume that $\Sigma$ is outer-minimizing with respect to the static metric $g$, and we need to show that it is also outer-minimizing with respect to the conformal metric $g_{-}= V^{\frac{4}{n-2}} g$. First notice that 

\begin{equation*}
    \inf_{M^{n}} V = V_{0},
\end{equation*}
Indeed, by the asymptotic expansion, $\inf_{M^{n}} V < 1$ when $m>0$. The smooth, bounded domain $\Omega_{\epsilon}= \{ V < 1 - \epsilon \} \subset M^{n}$ is therefore non-empty for small $\epsilon >0$. $V|_{\partial \Omega_{\epsilon}}$ equals $1 - \epsilon$ or $V_{0}$, and by the elliptic maximum principle $\inf_{\Omega_{\epsilon}} V = \min_{\partial \Omega_{\epsilon}} V$. Since $\inf_{\Omega_{\epsilon}} V < 1 - \epsilon$, we must have $\inf_{\partial \Omega_{\epsilon}} V = V_{0}$, and letting $\epsilon \rightarrow 0$ yields the conclusion.

Now, for any bounded domain $\Omega \subset M^{n}$ with $\Sigma \subset \partial \Omega$, the outer-minimizing assumption says

\begin{equation*}
    |\Sigma|_{g} \leq |\partial \Omega \setminus \Sigma|_{g}.
\end{equation*}
Then the areas with respect to $g_{-}$ obey

\begin{eqnarray*}
    |\Sigma|_{g_{-}} &=& V_{0}^{2\frac{n-1}{n-2}} |\Sigma|_{g} \leq \left( \min_{ \partial \Omega \setminus \Sigma} V^{2\frac{n-1}{n-2}} \right) |\Sigma|_{g} \\
    &\leq& (\min_{ \partial \Omega \setminus \Sigma} V^{2\frac{n-1}{n-2}}) |\partial \Omega \setminus \Sigma|_{g} \leq |\partial \Omega \setminus \Sigma|_{g_{-}} \nonumber,
\end{eqnarray*}
meaning $\Sigma$ is outer-minimizing with respect to $g_{-}$. Altogether, we can apply the inequality \eqref{conformal_inequality} to $\Sigma$. We obtain the inequality \eqref{static_level_set} by factoring $V_{0}$.

For the rigidity aspect of the theorem, the inequality \eqref{level_set} is equivalent to the inequality \eqref{inequality} on $(M^{n},g_{-})$, so since $V^{-}|_{\Sigma} = (V_{0})^{-1}$ is constant $(M^{n},g_{-})$ is isometric to a piece of the Schwarzschild manifold with mass $-m$. So

\begin{equation*}
    g = V_{-}^{\frac{4}{n-2}} g_{-} = V_{-m}^{\frac{4}{n-2}} g_{-m} = g_{m}.
\end{equation*}
\end{proof}

\section{Applications}

In this section, we prove uniqueness of Schwarzschild-stable static extensions of constant Bartnik data, of equipotential and Schwarzschild-stable CMC $2$-spheres, and of equipotential photon surfaces with small Einstein-Hilbert energy. We first show via local arguments that the second fundamental form of the boundary is pure trace in each of these situations. Then, we will use \thref{level_set_rig} to show saturation in inequality \eqref{inequality}.

\subsection{Static Metric Extensions}
In order to prove \thref{bartnik_data}, we first require an eigenvalue estimate on CMC spheres. Recall in \thref{stability} that we defined the constants

\begin{eqnarray*}
    r_{0}&=& \left(\frac{|\Sigma|}{w_{n-1}} \right)^{\frac{1}{n-1}}, \\
    m_{0}&=& \frac{r_{0}^{n-2}}{2} \left( 1- \frac{r_{0}^{2}}{(n-1)^{2}} H_{0}^{2} \right) 
\end{eqnarray*}
associated with a CMC hypersurface $\Sigma$ of mean curvature $H_{0}$. For $n=3$, Christodolou and Yau showed in \cite{quasi_local_remarks} showed that the Hawking mass $m_{H}(\Sigma)=m_{0}$ of a stable CMC $2$-sphere is non-negative provided $R_{g} \geq 0$. In fact, their choice of test function implies a more general fact: the Hawking mass of \textit{any} CMC $2$-sphere is bounded below by the smallest eigenvalue of its stability operator. In higher dimensions, a different choice of test function yields an analogous estimate for $n$-spheres that are intrinisically round. 
\begin{theorem}[First Eigenvalue for CMC Hyperspheres] \thlabel{eigenvalue}
Let $(M^{n},g)$ be a Riemannian manifold, and let $\Sigma = X_{0}(\mathbb{S}^{n-1}) \subset M^{n}$ be an immersed CMC sphere. Assume that either

\begin{itemize} 
\item $n=3$, OR
\item $n > 3$ and $\sigma_{ij}= r_{0}^{2} \sigma^{\mathbb{S}^{n-1}}_{ij}$ on $\Sigma$.
\end{itemize} 
Then the first eigenvalue of the stability operator $S_{\Sigma}$ of $\Sigma$ satisfies the upper bound
\begin{equation} \label{eigenvalue_est}
    \lambda_{1} (S_{\Sigma}) \leq  \frac{n(n-1) m_{0}}{r_{0}^{n}} - \min_{x \in \Sigma} \frac{R_{g}(x)}{2}.
\end{equation}
Furthermore, if equality holds, then $\Sigma$ is totally umbilical in $(M^{n},g)$ and the extrinsic scalar curvature $R_{g}$ is constant over $\Sigma$.
\end{theorem}

\begin{proof}
Let

\begin{equation*}
    X_{0}: (\mathbb{S}^{n-1},\sigma) \rightarrow (M^{n},g)
\end{equation*}
be an isometric immersion with $X_{0}(\mathbb{S}^{n-1})=\Sigma^{n-1}$. We consider an embedding $\widetilde{X}_{0}$ into the Euclidean sphere of radius $r_{0}$, that is

\begin{equation*}
    \widetilde{X}_{0}: (\mathbb{S}^{n-1}, \sigma) \rightarrow \partial B_{r_{0}} (0) \subset (\mathbb{R}^{n},\delta).
\end{equation*}
When $n>3$, we take $\widetilde{X}_{0}$ to be an isometric embedding given that $\sigma_{ij}=r_{0}^{2}\sigma_{ij}^{\mathbb{S}^{n-1}}$. When $n=3$, we instead take $\widetilde{X}_{0}$ to be a conformal diffeomorphism into $\partial B_{r_{0}}(0) \subset (\mathbb{R}^{3},\delta)$. In both cases, the coordinate functions $\widetilde{X}^{i}_{0}$ of $\widetilde{X}_{0}$ satisfy

\begin{eqnarray}
    \int_{\mathbb{S}^{n-1}} \widetilde{X}^{i}_{0} d\sigma &=& 0, \hspace{3cm} i=1, \dots, n \label{vanishing_coords}  \\ 
    \int_{\mathbb{S}^{n-1}} |\nabla \widetilde{X}^{i}_{0}|^{2} d\sigma &=& \frac{n-1}{n} w_{n-1} r_{0}^{n-1}. \label{dirichlet}
\end{eqnarray}
For $n>3$, these are simply computed on the Euclidean sphere. For $n=3$, Li and Yau showed in \cite{conformal_invariant} that \eqref{vanishing_coords} holds for any conformal map into $\partial B_{r}(0)$, and when $\widetilde{X}_{0}$ is a diffeomorphism \eqref{dirichlet} comes from conformal invariance of the Dirichlet energy. By the first equation, we can consider the stability operator \eqref{stability_operator} on $\phi= \widetilde{X}^{i}_{0}$, that is 

\begin{equation*}
    S_{\Sigma} \widetilde{X}^{i}_{0} = -\Delta_{\Sigma} \widetilde{X}^{i}_{0} - (|h|^{2} + \text{Ric}(\nu,\nu)) \widetilde{X}^{i}_{0}.
\end{equation*}
Multiplying by $\widetilde{X}^{i}_{0}$, integrating, and summing over $i=1, \dots, n$ yields
\begin{eqnarray} \label{stable}
    \sum_{i=1}^{n} \int_{\mathbb{S}^{n-1}} \widetilde{X}^{i}_{0} (S_{\Sigma} \widetilde{X}^{i}_{0}) d\sigma &=& (n-1)w_{n-1} r_{0}^{n-1} - \int_{\mathbb{S}^{n-1}} (|h|^{2} + \text{Ric}(\nu,\nu)) r_{0}^{2} d\sigma  \\
    &=& (n-1)w_{n-1}r_{0}^{n-1} - \int_{\mathbb{S}^{n-1}} \left(\frac{n}{2(n-1)} H^{2} + \frac{1}{2} |\mathring{h}|^{2} - \frac{1}{2} R_{\sigma} + \frac{1}{2} R_{g}\right) r_{0}^{2} d\sigma \nonumber \\
    &=& (n-1)w_{n-1}r_{0}^{n-1} - \frac{n}{2(n-1)} H_{0}^{2} w_{n-1}r_{0}^{n+1}  \nonumber \\
    & & - r_{0}^{2} \int_{\mathbb{S}^{n-1}} |\mathring{h}|^{2} d\sigma + \frac{r_{0}^{2}}{2} \int_{\mathbb{S}^{n-1}} R_{\sigma} d\sigma - \frac{r_{0}^{2}}{2} \int_{\mathbb{S}^{n-1}} R_{g} d\sigma,  \nonumber
\end{eqnarray}
where we have used the Gauss equation \eqref{gauss} to simplify the second term. Now, in our setting

\begin{equation*}
    R_{\sigma} = \begin{cases} 2K_{\sigma} & n=3, \\
    (n-1)(n-2) r_{0}^{-2} & n >3,
    \end{cases}
\end{equation*}
$K_{\sigma}$ being the intrinsic Gauss curvature of $\sigma$. When $n=3$ we can apply the Gauss-Bonnet theorem, and altogether for any $n$ we have that

\begin{equation*}
    \frac{r_{0}^{2}}{2} \int_{\mathbb{S}^{n-1}} R_{\sigma} d\sigma = \frac{(n-1)(n-2)}{2} w_{n-1} r_{0}^{n-1}.
\end{equation*}
Substituting back into \eqref{stable} gives 

\begin{eqnarray}
    \sum_{i=1}^{n} \int_{\mathbb{S}^{n-1}} \widetilde{X}^{i}_{0} (S_{\Sigma} \widetilde{X}^{i}_{0}) d\sigma &=& \frac{n(n-1)}{2} w_{n-1} r_{0}^{n-1} - \frac{n}{2(n-1)} H_{0}^{2} w_{n-1} r_{0}^{n+1} \nonumber \\
    & &- \frac{r_{0}^{2}}{2} \int_{\mathbb{S}^{n-1}} |\mathring{h}|^{2} d\sigma - \frac{r_{0}^{2}}{2} \int_{\mathbb{S}^{n-1}} R_{g} d\sigma \nonumber \\
    &=& \frac{n(n-1)}{2} w_{n-1} r_{0}^{n-1} \left[ 1 - \frac{1}{(n-1)^{2}} H_{0}^{2} r_{0}^{2}  \right]\label{stable2} \\ 
    & & - \frac{r_{0}^{2}}{2} \int_{\mathbb{S}^{n-1}} |\mathring{h}|^{2} d\sigma - \frac{r_{0}^{2}}{2} \int_{\mathbb{S}^{n-1}} R_{g} d\sigma \nonumber \\
    &\leq& \left(\frac{n(n-1)m_{0}}{r_{0}^{n}} - \frac{1}{2} \min_{x \in \Sigma} R_{g}(x)\right)  \left(\sum_{i=1}^{n} \int_{\mathbb{S}^{n-1}} (\widetilde{X}^{i}_{0})^{2} d\sigma \right). \nonumber
\end{eqnarray}
Altogether, there is a $j \in \{ 1, \dots, n \}$ such that

\begin{equation*}
    \int_{\mathbb{S}^{n-1}} \widetilde{X}^{j}_{0} (S_{\Sigma} \widetilde{X}^{j}_{0}) d\sigma \leq \left(\frac{n(n-1)m_{0}}{r_{0}^{n}} - \frac{1}{2} \min_{x \in \Sigma} R_{g}(x) \right) \int_{\mathbb{S}^{n-1}} (\widetilde{X}^{j}_{0})^{2} d\sigma,
\end{equation*}
and from \eqref{stable2} this inequality is strict unless $|\mathring{h}|=0$ and $R_{g}|_{\Sigma} = \min_{\Sigma} R_{g}(x)$.

\end{proof}
According to \thref{eigenvalue}, a Schwarzschild-stable hypersurface $\Sigma^{n-1}$ in a static manifold $(M^{n},g)$ with the Bartnik data \eqref{bartnik_data} is totally umbilical. The Gauss-Codazzi equations are much more tractable given this fact, and the induced metric allows us to deduce that $\Sigma^{n-1}$ is equipotential when $m_{0} \geq 0$ (recall that $m_{0}$ equals the mass of the Schwarzschild extension).

\begin{lemma} \thlabel{static_local}
Let $(M^{n},g)$ be a static Riemannian manifold with boundary $\partial M = \Sigma^{n-1} \cong \mathbb{S}^{n-1}$. Suppose on $\Sigma^{n-1}$ that

\begin{eqnarray*}
    \sigma_{ij} &=& r_{0}^{2} \sigma_{ij}^{\mathbb{S}^{n-1}}, \\
    H_{\Sigma} &=& H_{0},
\end{eqnarray*}
for constants $r_{0}, H_{0} > 0$. Suppose further that $\Sigma^{n-1}$ is Schwarzschild-stable in $(M^{n},g)$. Then the following hold:
\begin{enumerate} [label=(\alph*)]
    \item If $H_{0} < (n-1) r_{0}^{-1}$, then
    \begin{eqnarray*}
    V|_{\Sigma} &=& V_{0}, \\
    \frac{\partial V}{\partial \nu} &=& \rho_{0}, 
\end{eqnarray*}
for constants $V_{0}$, $\rho_{0}$.
\item If $H_{0}=(n-1) r_{0}^{-1}$ and either the maximum or the minimum of $V$ in $\overline{M^{n}}$ occurs on $\Sigma^{n-1}$, then $V$ is constant in $M^{n}$.
\end{enumerate} 
\end{lemma}
\begin{proof}
Since $R_{g}=0$, Schwarzschild stability implies that $\Sigma^{n-1}$ is totally umbilical in view of the previous proposition. Thus the left-hand side of the Codazzi equation \eqref{codazzi} identically vanishes on $\Sigma^{n-1}$, and we are left with

\begin{equation*}
    0= \text{Rc}(X, \nu) = \frac{1}{V} (X(\frac{\partial V}{\partial \nu}) -h(\nabla_{\Sigma} V, X)) = \frac{1}{V} X( \frac{\partial V}{\partial \nu} - \frac{H_{0}}{n-1} V).
\end{equation*}
Hence
\begin{equation} \label{linear}
    \frac{\partial V}{\partial \nu} (y) = \frac{H_{0}}{n-1} V(y) + \beta, \hspace{2cm} y \in \Sigma^{n-1},
\end{equation}
for a constant $\beta \in \mathbb{R}$. To determine the value of $\beta$, we use the static equation on surfaces \eqref{surface_static} and find

\begin{equation*}
    \Delta_{\Sigma} V + \frac{H_{0}^{2}}{n-1} V + \beta H_{0}= - \text{Ric}(\nu,\nu) V.
\end{equation*}
For $|\mathring{h}|=0$ and $R_{g}=0$, we compute

\begin{eqnarray*}
\frac{H^{2}}{n-1} + \text{Ric}(\nu,\nu) &=& \frac{n}{2(n-1)} H_{0}^{2} - \frac{(n-1)(n-2)}{r_{0}^{2}} \\
&=& \frac{n-1}{r_{0}^{2}} - \frac{n(n-1)m_{0}}{r_{0}^{n-1}}, 
\end{eqnarray*}
and so 
\begin{equation} \label{surface_laplacian}
    -\Delta_{\Sigma} V = \left( \frac{n-1}{r_{0}^{2}} - \frac{n(n-1)m_{0}}{r_{0}^{n-1}} \right) V + \beta H_{0}.
\end{equation}
Integrating implies

\begin{equation} \label{beta}
    \beta= -H_{0}^{-1} \left(\frac{n-1}{r_{0}^{2}} - \frac{n(n-1)m_{0}}{r_{0}^{n-1}} \right) \overline{V},
\end{equation}
where $\overline{V}= \frac{1}{|\Sigma|} \int_{\Sigma} V d\sigma$ is the average value of $V$. By combining equations \eqref{surface_laplacian} and \eqref{beta}, we find for the function $f= V - \overline{V}$ that

\begin{equation} \label{V_avg}
    -\Delta_{\Sigma} f = \left(\frac{n-1}{r_{0}^{2}} - \frac{n(n-1)m_{0}}{r_{0}^{n-1}} \right) f.
\end{equation}
Let us first consider the case $H_{0} < (n-1) r_{0}^{-1}$. Since $\Delta_{\Sigma} = \Delta_{r_{0}^{2} \sigma^{\mathbb{S}^{n-1}}}$, the first eigenvalue of $\Delta_{\Sigma}$ is $\lambda_{1}(-\Delta_{\Sigma}) = \frac{n-1}{r_{0}^{2}}$. But

\begin{equation*}
    m_{0} = \frac{r_{0}^{n-2}}{2} \left(1- \frac{r_{0}^{2}}{(n-1)^{2}} H_{0}^{2} \right) > 0,
\end{equation*}
which means that the only solution to \eqref{V_avg} is the trivial one. It immediately follows that both $V$ and $\frac{\partial V}{\partial \nu}$ are constant, proving item $(a)$.

Next we consider the case $H_{0}=(n-1) r_{0}^{-1}$. Inputting \eqref{beta} into equation \eqref{linear}, we find

\begin{eqnarray*}
    \frac{\partial V}{\partial \nu} &=& \frac{H_{0}}{n-1} V - H_{0}^{-1} ( \frac{n-1}{r_{0}^{2}} - \frac{n(n-1) m_{0}}{r_{0}^{n-1}}) \overline{V} \\
    &=& \frac{H_{0}}{n-1} (V - \frac{(n-1)^{2}}{H^{2}_{0} r_{0}^{2}} \overline{V}) = \frac{H_{0}}{n-1} f. 
\end{eqnarray*}
Since $f$ solves \eqref{V_avg}, we must have

\begin{equation} \label{zero}
    \int_{\Sigma} \frac{\partial V} {\partial \nu} d\sigma = 0.
\end{equation}
Now, suppose that $\max_{\overline{M}^{n}} V = V(y_{0})$ for some $y_{0} \in \Sigma^{n-1}$. If $V$ is non-constant in $M^{n}$, then

\begin{equation*}
    \frac{\partial V}{\partial \nu} (y_{0}) < 0
\end{equation*}
by the Hopf Lemma. Since $V(y_{0}) = \max_{\Sigma^{n-1}} V$, we must also have $\frac{\partial V}{\partial \nu} \leq \frac{\partial V}{\partial \nu} (y_{0})$ on $\Sigma^{n-1}$ by \eqref{linear}. This contradicts \eqref{zero}, and so $V$ must be constant in $M^{n}$. The argument is identical if $\min_{\overline{M}^{n}} V$ occurs on $\Sigma^{n-1}$.
\end{proof}

When $H_{0} < (n-1)r_{0}$, the equipotential Minkowski inequality \eqref{static_level_set} for asymptotically flat static systems provides an upper bound on the constant $V_{0}$. This also gives us an upper bound the left-hand side of the inequality \eqref{inequality} which immediately implies saturation for this Bartnik data.
 
\begin{proof}[Proof of \thref{bartnik}]
First, let us address the case $H_{0}=(n-1)r_{0}^{-1}$. Since $V$ is asymptotic to $1$, either its maximum or its minimum occurs on $\Sigma^{n-1}$. Thus $V \equiv 1$ on $M^{n}$ by \thref{static_local}(b). The left-hand side of the Minkowski inequality then becomes

\begin{equation*}
    \frac{1}{(n-1)w_{n-1}}\int_{\Sigma} H V d\sigma + 2m = \frac{1}{(n-1) w_{n-1}} \int_{\Sigma} H_{0} d\sigma
    = r_{0}^{n-2} = \left(\frac{|\Sigma|}{w_{n-1}} \right)^{\frac{n-2}{n-1}}.
\end{equation*}
Therefore, $(M^{n},g) \cong (\mathbb{R}^{n} \setminus B_{r_{0}}(0), \delta)$ by \thref{cmc_rigidity} (we stress again that we do not need $m > 0$ for this theorem).

Next, we address $H_{0} < (n-1) r_{0}^{-1}$. With the help of the static equation once again, we compute the ADM mass of $(M^{n},g)$ as

 \begin{eqnarray}
     m &=& \frac{1}{(n-2)w_{n-1}} \int_{\Sigma} \frac{\partial V}{\partial \nu} d\sigma =  \frac{1}{(n-2) w_{n-1}} \int_{\Sigma} -\text{Ric}(\nu,\nu) \frac{V}{H} d \sigma \nonumber \\
     &=& \frac{V_{0}}{2 (n-2) w_{n-1} H_{0}} \int_{\Sigma} \left((n-2)(n-1) r_{0}^{-2} -\frac{n-2}{(n-1)} H_{0}^{2} \right) d \sigma \nonumber \\
     &=& \frac{V_{0}}{2 H_{0}} r_{0}^{n-1} \left( (n-1) r_{0}^{-2} -\frac{1}{n-1} H_{0}^{2} \right) = \frac{(n-1)V_{0}}{H_{0} r_{0}} m_{0} > 0. \nonumber
 \end{eqnarray}
 Since $\Sigma$ is outer-minimizing with positive ADM mass, we may apply \thref{level_set_rig}. Inequality \eqref{static_level_set} gives
 \begin{equation} \label{upper_bound}
     V_{0} \leq \frac{1}{(n-1)w_{n-1}} \left(\frac{|\Sigma|}{w_{n-1}} \right)^{\frac{2-n}{n-1}} \int_{\Sigma} H_{0} d \sigma = \frac{1}{n-1} r_{0}H_{0}.
 \end{equation}
 Therefore, we may bound the left-hand side of the Minkowski inequality \eqref{inequality} above:

 \begin{eqnarray} \label{upper_bound2}
     \frac{1}{(n-1)w_{n-1}} \int_{\Sigma} VH d\sigma + 2 m &=& \frac{1}{n-1} V_{0} H_{0} r_{0}^{n-1} + 2 \frac{(n-1) V_{0}}{H_{0}r_{0}} m_{0} \nonumber \\
     &\leq& \frac{1}{(n-1)^{2}} H_{0}^{2} r_{0}^{n} + 2m_{0} \\
     &=& r_{0}^{n-2} = \left(\frac{|\Sigma|}{w_{n-1}} \right)^{\frac{n-2}{n-1}}. \nonumber
 \end{eqnarray}
 This means that equality is achieved in \eqref{inequality}. The conclusion again follows from \thref{cmc_rigidity}.
\end{proof}

\subsection{Equipotential CMC Spheres}
We would like to show that the Schwarzschild manifold is the only asymptotically flat static $3$-manifold containing a CMC sphere that is equipotential and Schwarzschild stable. The key ingredient is the following corollary of \thref{eigenvalue}.

\begin{corollary}
Let $(M^{3},g)$ be a Riemannian manifold with scalar curvature $R_{g} \geq 0$, and let $\Sigma^{2}=X_{0}(\mathbb{S}^{2}) \subset (M^{3},g)$ be an immersed CMC sphere. If $\Sigma^{2}$ is Schwarzschild stable in $(M^{3},g)$, then it is umbilical, i.e. $h_{ij}= \frac{H_{0}}{2} \sigma_{ij}$, and $R_{g}$ vanishes over $\Sigma^{2}$.
\end{corollary}

\begin{proof}[Proof of \thref{equi_cmc}]
In view of this corollary, we have on $\Sigma^{2}$ that

\begin{eqnarray*}
    V|_{\Sigma} &=& V_{0}, \\
    H_{\Sigma} &=& H_{0}, \\
    h_{ij} &=& \frac{H_{0}}{2} \sigma_{ij}.
\end{eqnarray*}
Therefore,

\begin{eqnarray}
    m &=& \frac{1}{w_{2}} \int_{\Sigma} \frac{\partial V}{\partial \nu} d\sigma = -\frac{V_{0}}{H_{0}w_{2}} \int_{\Sigma} \text{Ric}(\nu,\nu) d\sigma \nonumber \\
    &=& \frac{V_{0}}{H_{0}w_{2}} \int_{\Sigma} \left(K_{\sigma} - \frac{1}{4} H_{0}^{2} \right) d\sigma \label{cmc_mass} \\
    &=& \frac{V_{0}}{H_{0} w_{2}} \left( w_{2} - \frac{w_{2}}{4} H_{0}^{2} r_{0}^{2} \right) = 2\frac{ V_{0}}{H_{0} r_{0}} m_{0} > 0, \nonumber
\end{eqnarray}
where we have used the Gauss-Bonnet theorem once again. Following the static extension proof, we apply the equipotential Minkowski inequality \eqref{static_level_set} and find

\begin{eqnarray*}
    \frac{1}{2w_{2}} \int_{\Sigma} V H d\sigma + 2m &=& \frac{V_{0}}{H_{0}} \left( \frac{1}{2} H_{0}^{2} r_{0}^{2} + 4 m_{0} \right) = 2 \frac{V_{0}}{H_{0}} \\
    &\leq& \left(\frac{|\Sigma|}{w_{2}} \right)^{\frac{1}{2}}.
\end{eqnarray*}
The conclusion follows.
\end{proof}
\subsection{Photon Surfaces}
Now, we address \thref{photon_surface}. We must first consider the Lorentzian picture to understand the geometry of $\Sigma^{2} = P^{3} \cap \{ t = t_{0} \} \subset M^{3} \times \{ t=t_{0} \}$. In any Lorentzian manifold of the form \eqref{static_product} such a time slice is necessarily umbilical. This fact does not depend on the static equations, and since photon surfaces are of interest in a broader class of spacetimes we will present a more general statement here.

\begin{proposition}
Let $(L^{n+1},\overline{g})$ be a Lorentzian manifold with the warped product structure,

\begin{eqnarray*}
    L^{n+1} &=& M^{n} \times \mathbb{R}, \\
    \overline{g} &=& -f(x)^{2} dt^{2} + g, \hspace{1cm} x \in M^{n},
\end{eqnarray*}
for some base manifold $M^{n}$ with Riemannian metric $g$ and function $f: M^{n} \rightarrow \mathbb{R}^{+}$. Let $P^{n} \subset (L^{n+1}, \overline{g})$ be a photon surface. For each $t_{0} \in \mathbb{R}$, the time slice $\Sigma^{n-1} = P^{n} \cap \{ t = t_{0} \} \subset M^{n} \times \{ t_{0} \}$ is totally umbilical in $(M^{n} \times \{ t_{0} \},g)$, i.e. 

\begin{equation*}
    h_{ij} = \frac{H}{n-1} \sigma_{ij}.
\end{equation*}
\end{proposition}
\begin{proof}
 The proof is more or less the same as the one in Section 4 of \cite{photon_surface_equipotential}, but we present it here for the convenience of the reader. We show first that $M^{n} \times \{ t_{0}\}$ is totally geodesic in $(L^{n+1}, \overline{g})$. Denote the metric connection on $(L^{n+1}, \overline{g})$ by $\overline{\nabla}$, and its induced connection of $M^{n} \times \{ t=t_{0}\}$ by $\nabla$. We work in coordinates $(x^{1},x^{2}, \dots, x^{n},t)$ on $L^{n+1}$ for local coordinates $(x^{1},x^{2}, \dots,x^{n})$ on $M^{n}$. Then

\begin{eqnarray}
    0 = \partial_{t} g_{ij} &=& \overline{g}(\overline{\nabla}_{\partial_{t}} \partial_{i}, \partial_{j}) + \overline{g}(\overline{\nabla}_{\partial_{t}} \partial_{j}, \partial_{i}) \hspace{1.5cm} i,j=1, \dots, n \label{christoffels} \\
    &=& \overline{g}_{jk} \overline{\Gamma}^{k}_{ti} + \overline{g}_{ik} \overline{\Gamma}^{k}_{tj} \nonumber
\end{eqnarray}
Now, the timelike unit vector field

\begin{equation*}
    n(x,t) = f^{-1}(x) \partial_{t} \in \Gamma(T L^{n+1})
\end{equation*}
is everywhere orthogonal to $ M \times \{ t_{0} \}$, and so the second fundamental form $k_{ij}$ of $M \times \{ t_{0} \} \subset (L^{n+1}, \bar{g})$ is

\begin{eqnarray*}
    k_{ij}=\overline{g}( \overline{\nabla}_{\partial_{i}} n, \partial_{j}) &=& \partial_{i} f^{-1}(x) \overline{g}_{tj} + f^{-1}(x) \overline{g}(\partial_{t}, \overline{\nabla}_{\partial_{i}} \partial_{j}) \hspace{1cm} i,j=1,\dots,n \\
    &=& f^{-1}(x) \partial_{i} \overline{g}_{tj} - f^{-1}(x) \overline{g}(\overline{\nabla}_{\partial_{i}} \partial_{t}, \partial_{j}) \\
    &=& - f^{-1}(x) \overline{g}_{jk} \overline{\Gamma}^{k}_{ti}.
\end{eqnarray*}
Then 

\begin{equation*}
    k_{ij}= k_{ji} = - k_{ij},
\end{equation*}
where the last equality comes from \eqref{christoffels}. Altogether, $M \times \{ t_{0} \}$ is totally geodesic in $(L^{n+1},\bar{g})$.

Now, denote the unit normal of $P^{n} \subset (L^{n+1}, \overline{g})$ by $\overline{\nu}$. Notice that $\overline{\nu}$ must be spacelike given the signature of $\overline{g}$. Also denote by $\nu$ the unit normal of $\Sigma^{n-1} \subset (M^{n} \times \{ t_{0} \}, g)$. $\Sigma^{n-1} \subset L^{n+1}$ is a codimension-$2$ submanifold with normal bundle spanned by the vector fields $\nu$ and $n$. In particular, for each $y \in \Sigma^{n-1}$ there exist constants $a(y)$ and $b(y)$ such that

\begin{equation*}
    \overline{\nu} = a(y) n + b(y) \nu \in T_{x} L^{n+1}.
\end{equation*}
Notice that $b(y) \neq 0$ anywhere because $\overline{\nu}$ is spacelike. For local coordinates $(y^{1}, \dots, y^{n-1})$ on $\Sigma^{n-1}$, we evaluate the second fundamental form $\overline{h}$ of $P^{n}$ on the coordinate vector fields $\partial_{i} \in T_{y} \Sigma^{n-1}$. By umbilicity of $\overline{h}$,

\begin{eqnarray*}
    \frac{\overline{H}}{n} \sigma_{ij} = \overline{h}_{ij} &=& - \overline{g} ( \overline{\nu}, \overline{\nabla}_{\partial_{i}} \partial_{j}) =-a(y) \overline{g} ( n,\overline{\nabla}_{\partial_{i}} \partial_{j}) - b(y) \overline{g} ( \nu, \overline{\nabla}_{\partial_{i}} \partial_{j}) \\
    &=& a(x) k_{ij} -b(y) \overline{g} ( \nu, \overline{\nabla}_{\partial_{i}} \partial_{j}) = - b(y) g(\nu, \nabla_{\partial_{i}} \partial_{j}) \\
    &=& b(y) h_{ij}. 
\end{eqnarray*}
Conclude $h_{ij}(y) = \frac{\overline{H}}{n b(y)} \sigma_{ij}$.
\end{proof}
\begin{corollary}
Let $(L^{n+1},\overline{g})= (M^{n} \times \mathbb{R}, -V(x)^{2}dt^{2} + g)$ be a static spacetime, and let $P^{n} \subset (L^{n+1}, \overline{g})$ be a photon surface. Suppose there is a $t_{0} \in \mathbb{R}$ such that on the time slice $\Sigma^{n-1}=P^{n} \times \{ t=t_{0} \} \subset M^{n} \times \{ t_{0} \}$ we have

\begin{eqnarray*}
    V|_{\Sigma}&=& V_{0}, \\
    \frac{\partial V}{\partial \nu} &=& \rho_{0}.
\end{eqnarray*}
for constants $V_{0}$, $\rho_{0}$. Then $\Sigma^{n-1} \subset (M^{n} \times \{ t_{0} \},g)$ is CMC.
\end{corollary}
\begin{proof}
By the previous proposition, $\Sigma^{n-1}$ is also umbilical in $M^{n} \times \{ t_{0} \}$. For $X \in T_{x} \Sigma$

\begin{equation*}
    \text{Ric}(X, \nu) = \frac{1}{V_{0}} \nabla^{2} V(X, \nu) = \frac{1}{V_{0}} (X(\nu(V)) - h(X, \nabla_{\Sigma} V)) = 0.
\end{equation*}
On the other hand, the left-hand side of the Codazzi equation \eqref{codazzi} equals $\frac{2-n}{n-1} X(H)$. Conclude that $\nabla_{\Sigma} H =0$.
\end{proof}

\begin{proof}[Proof of \thref{photon_surface}]
Given

\begin{eqnarray*}
    V|_{\Sigma} &=& V_{0}, \\
    H_{\Sigma} &=& H_{0}, \\
    h_{ij} &=& \frac{H_{0}}{n-1} \sigma_{ij},
\end{eqnarray*}
on $\Sigma^{n-1} = P^{n} \cap \{ t= t_{0} \} \subset (M^{n} \times \{ t_{0} \},g)$, we follow the proofs of the previous two uniqueness theorems. Once again from the mass formula \eqref{mass_formula}, the static equation \eqref{surface_static}, and the Gauss equation \eqref{gauss}, we get

\begin{eqnarray*}
    m &=& \frac{1}{(n-2)w_{n-1}} \int_{\Sigma} \frac{\partial V}{\partial \nu} d\sigma =  \frac{1}{(n-2)w_{n-1}} \int_{\Sigma} -\text{Ric}(\nu,\nu) \frac{V}{H} d\sigma  \\
    &=& \frac{V_{0}}{2(n-2) w_{n-1} H_{0}} \int_{\Sigma} \left( R_{\sigma} - \frac{n-2}{n-1} H_{0}^{2} \right) d\sigma = \frac{V_{0}} {2(n-2) w_{n-1} H_{0}} \left(w_{n}^{\frac{n-3}{n-1}} E(\Sigma) r_{0}^{n-3} - \frac{n-2}{n-1} H_{0}^{2} \right) \\
    &\leq& \frac{V_{0}}{2H_{0}} r_{0}^{n-1} \left( (n-1)r_{0}^{-2} - \frac{1}{n-1} H_{0}^{2} \right) =\frac{(n-1)V_{0}}{H_{0}r_{0}} m_{0},
\end{eqnarray*}
with the upper bound coming from $E(\Sigma) \leq E(\mathbb{S}^{n-1}) = (n-2)(n-1) w_{n-1}^{\frac{2}{n-1}}$. Applying \eqref{upper_bound} and \eqref{upper_bound2} yields saturation once again, and so the region outside $P^{n} \cap \{ t= t_{0} \}$ is rotationally symmetric. Since an open subset of $(M^{n},g)$ isometric embeds into the maximal Schwarzschild manifold \eqref{schwarzschild}, the conclusion follows from the analyticity of static metrics, see for example \cite{static_analytic} and \cite{static_analytic_chrusciel}. 



\end{proof}

\section{Hawking Mass Comparison}

In this section, we prove \thref{hawking_comparison}. Given that $\Sigma$ is outer-minimizing with respect to the metric $g_{-}$, the inequality \eqref{comparison} is a simple consequence of \thref{conformal_rig}.

\begin{proof}[Proof of \thref{hawking_comparison}]
Applying the inequality \eqref{conformal_inequality} from \thref{conformal_rig} and H\"older's inequality for $p=\frac{n-1}{n-2}$, we have that

\begin{eqnarray}
    \int_{\Sigma} V^{2} d \sigma &\leq& |\Sigma|^{\frac{1}{n-1}} \left(\int_{\Sigma} V^{2 \frac{n-1}{n-2}} d\sigma \right)^{\frac{n-2}{n-1}} \leq \frac{1}{n-1} \left(\frac{|\Sigma|}{w_{n-1}} \right)^{\frac{1}{n-1}} \int_{\Sigma} VH d\sigma \label{cauchy_schwarz} \\
    &\leq& \frac{1}{n-1} \left(\frac{|\Sigma|}{w_{n-1}} \right)^{\frac{1}{n-1}} \left(\int_{\Sigma} V^{2} d\sigma \right)^{\frac{1}{2}} \left( \int_{\Sigma} H^{2} d\sigma \right)^{\frac{1}{2}}. \nonumber
\end{eqnarray}
Thus

\begin{equation} \label{holder}
    \left(\frac{|\Sigma|}{w_{n-1}} \right)^{-\frac{1}{n-1}} \int_{\Sigma} VH d\sigma \leq \left(\frac{|\Sigma|}{w_{n-1}} \right)^{-\frac{1}{n-1}} \left(\int_{\Sigma} V^{2} d\sigma \right)^{\frac{1}{2}} \left(\int_{\Sigma} H^{2} d\sigma \right)^{\frac{1}{2}} \leq \frac{1}{n-1} \int_{\Sigma} H^{2} d\sigma.
\end{equation}
This implies $Q(\Sigma) \geq m_{H}(\Sigma)$ for $n=3$. If equality holds then $V|_{\Sigma}= V_{0}$ and $H_{\Sigma} = H_{0}$ are each constant by \eqref{cauchy_schwarz}. But by \eqref{holder}, we must also have

\begin{equation*}
    V_{0} = \frac{1}{n-1} \left(\frac{|\Sigma|}{w_{n-1}} \right)^{\frac{1}{n-1}} H_{0},
\end{equation*}
implying equality in inequality \eqref{static_level_set}.
\end{proof}

\appendix

\printbibliography[title={References}]

\begin{center}
\textnormal{ \large National Center for Theoretical Sciences, Mathematics Division \\
National Taiwan University \\
Taipei City, Taiwan 10617\\
e-mail: bharvie@ncts.ntu.edu.tw}\\
\end{center}
\vspace{1cm}
\begin{center}
\textnormal{ \large Department of Applied Mathematics \\
National Yang-Ming Chiao Tung University \\
Hsinchu, Taiwan 30010 \\
\vspace{0.5cm}
National Center for Theoretical Sciences, Mathematics Division \\
National Taiwan University \\
Taipei City, Taiwan 10617\\
e-mail: ykwang@math.nycu.edu.tw}\\
\end{center}

\end{document}